\documentclass[12pt]{article}

\usepackage{graphics,graphicx,fullpage,natbib,multirow}
\usepackage{amsmath,amssymb,verbatim,epsfig}
\usepackage[dvipsnames,usenames]{color}
\usepackage{breqn}
\usepackage{float}
\usepackage{csvsimple}
\usepackage{siunitx}
\usepackage[table]{xcolor}
\usepackage{authblk}

\newtheorem{proposition}{Proposition}

\newtheorem{defin}{\bf Definition}

\newenvironment{proof}{\noindent{\bf Proof}}{$\diamond$}

\DeclareMathOperator{\sgn}{sgn}

\def\no{\mbox{N}}

\def\un{\mbox{Un}}

\def\P{\mbox{P}}

\def\d{\mbox{d}}

\def\data{\mbox{data}}

\def\bu{{\bf u}}
\def\bv{{\bf v}}

\def\bX{{\bf X}}
\def\bY{{\bf Y}}

\newcommand{\btheta}{\boldsymbol{\theta}}

\begin{document}

\baselineskip=24pt

\title{\bf A Bayesian semiparametric Archimedean copula}
\author{{\sc Ricardo Hoyos and Luis Nieto-Barajas} \\
{\sl Department of Statistics, ITAM, Mexico} \\
{\small {\tt ricardo.hoyos@itam.mx {\rm and} lnieto@itam.mx}}\\}
\date{}

\maketitle

\begin{abstract}
An Archimedean copula is characterised by its generator. This is a real function whose inverse behaves as a survival function. We propose a semiparametric generator based on a quadratic spline. This is achieved by modelling the first derivative of a hazard rate function, in a survival analysis context, as a piecewise constant function. Convexity of our semiparametric generator is obtained by imposing some simple constraints. The induced semiparametric Archimedean copula produces Kendall's tau association measure that covers the whole range $(-1,1)$. Inference on the model is done under a Bayesian approach and for some prior specifications we are able to perform an independence test. Properties of the model are illustrated with a simulation study as well as with a real dataset. 
\end{abstract}

\noindent {\sl Keywords}: Archimedean copula, Bayes nonparametrics, piecewise constant, survival analysis, quadratic spline.

\noindent {\sl AMS Classification}: 60E05 $\cdot$ 62G05 $\cdot$ 62N86.

\section{Introduction}
\label{sec:intro}

Let $\varphi(\cdot)$ be a continuous, strictly decreasing function from $[0,1]$ to $[0,\infty)$ such that $\varphi(1)=0$. Let $\varphi^{-1}(\cdot)$ be the inverse or the pseudo-inverse of $\varphi$, where the latter is defined as zero for $t>\varphi(0)$. If $\varphi(t)\to\infty$ as $t\to 0^+$ the generator is called strict. For instance, $\varphi(t)=-\log(t)$, is an example of a strict generator. An Archimedean copula $C(u,v)$ with generator $\varphi$ is a function from $[0,1]^2$ to $[0,1]$  defined as 
\begin{equation}
\label{eq:arqc}
C(u,v)=\varphi^{-1}\left(\varphi(u)+\varphi(v)\right). 
\end{equation}
A further requirement for \eqref{eq:arqc} to be well defined is that $\varphi$ must be convex \citep[e.g.][]{nelsen:06}. 

There are many properties that characterize Archimedean copulas, for instance, they are symmetric, associative and their diagonal section $C(u,u)$ is always less than $u$ for all $u\in(0,1)$. Generators $\varphi(\cdot)$ are usually parametric families defined by a single parameter. Most of them are summarised in \cite[][Table 4.1]{nelsen:06} and few of them are also included in Table \ref{tab:parcopulas}. 

Association measures induced by Archimedean copulas are a function of the generator. For instance, Kendall's tau becomes 
\begin{equation}
\label{eq:ktau}
\kappa_\tau=1+4\int_0^1\frac{\varphi(t)}{\varphi^{\prime}(t^+)}\d t,
\end{equation}
where $\varphi^{\prime}(t^+)$ denotes the right derivative of $\varphi$ at $t$. 

In this work we propose a Bayesian semiparametric generator defined through a quadratic spline. Within a survival analysis context, we model the first derivative of a hazard rate function with a piecewise constant function. The hazard rate and the cumulative hazard functions become linear and quadratic continuous functions, respectively. The induced survival function is used as an inverse generator for an Archimedean copula. Convexity constraints are properly addressed and inference on the model is done under a Bayesian approach. 

Other studies on semiparametric generators for Archimedean copulas can be found in \cite{genest&rivest:93} where their model is based on an  empirical Kendall's process. A new approach and extensions of this latter methodology can be found in \cite{genest&al:11}. In \cite{guillote&perron:15} the model arises from the one-to-one correspondence between an Archimedean generator and a distribution function of a nonnegative random variable. In particular they use a mixture of P\'olya trees as a prior for the corresponding distribution function under a Bayesian nonparametric approach. In a work more related to ours, \cite{vandenhende&lambert:05} use the relationship between quantile functions and Archimedean generators to define a semiparametric generator by supplementing a parametric generator with $n+1$ dependence parameters. Differing to their work, our model is not based on any parametric generator and the Kendall's tau can take values on the whole interval $(-1,1)$. 

The contents of the rest of the paper is as follows. In Section \ref{sec:model} we present our proposal and characterise its properties. In Section \ref{sec:post} we provide details of how to make posterior inference under a Bayesian approach. In Section \ref{sec:illust} we illustrate the performance of our model with a simulation study as well as with a real data set. We conclude with some remarks in Section \ref{sec:concl}.

Before proceeding we introduce notation: $\un(a,b)$ denotes a continuous uniform density on the interval $(a,b)$; and $\no(\mu,\sigma^2)$ denotes a normal density with mean $\mu$ and variance $\sigma^2$.

\section{Model}
\label{sec:model}

To elicit our proposal we noticed that $\varphi^{-1}$ is a decreasing function from $[0,\infty)$ to $[0,1]$, so it behaves as a survival function, in a failure time data analysis context \citep[e.g.][]{klein&moeschberger:03}. The idea is to propose a semi/non parametric form for the inverse generator $\varphi^{-1}$ by using survival analysis ideas. For that we recall some basic definitions. 

Let $h(t)$ be a nonnegative function with domain in $[0,\infty)$ such that $H(t)=\int_0^t h(s)\d s$ satisfies $H(t)\to\infty$ as $t\to\infty$. Then $S(t)=\exp\{-H(t)\}$ is a decreasing function from $[0,\infty)$ to $[0,1]$, so it behaves like an inverse generator $\varphi^{-1}(t)$. In a survival analysis context, functions $h(\cdot)$, $H(\cdot)$ and $S(\cdot)$ are the hazard rate, cumulative hazard and survival functions, respectively. 

In particular, if $h(t)=\theta$, i.e. constant for all $t$, then $S(t)=e^{-\theta t}$. If we take $\varphi(t)^{-1}=e^{-\theta t}$, then $\varphi(t)=-(\log t)/\theta$. Using \eqref{eq:arqc} we obtain that the resulting copula $C(u,v)=uv$ is the independence copula, and what is interesting, is that it does not depend on $\theta$. 

\subsection{Main proposal}
\label{subsec:main}

Using the previous ideas we construct a semiparametric generator in the following way. We first consider a partition of size $K$ of the positive real line, with interval limits given by $0=\tau_0<\tau_1<\cdots<\tau_K=\infty$. Then, we define the first derivative of the hazard rate, as a piecewise constant function of the form 
\begin{equation}
\label{eq:hp}
h'(t)=\sum_{k=1}^K \theta_k I(\tau_{k-1}<t\leq\tau_k),
\end{equation}
where  $\theta_K\equiv 0$.
We recover the hazard rate function as $h(t)=\int_0^t h'(s)\d s+\theta_0$, where $h(0)=\theta_0>0$ is an initial condition. Using \eqref{eq:hp}, the hazard rate becomes a piecewise linear function of the form
\begin{equation}
\label{eq:h}
h(t)=\sum_{k=1}^K \left(A_k+\theta_k t\right) I(\tau_{k-1}<t\leq\tau_k),
\end{equation}
where $A_1=\theta_0$ and $A_k=\theta_0+\sum_{j=1}^{k-1}(\theta_j-\theta_{j+1})\tau_j$, for $k=2,\ldots,K$. 

Integrating now the hazard function \eqref{eq:h}, the cumulative hazard  is a piecewise quadratic function given by 
\begin{equation}
\label{eq:H}
H(t)=\sum_{k=1}^K \left(B_k+A_k t+\frac{\theta_k}{2} t^2\right) I(\tau_{k-1}<t\leq\tau_k),
\end{equation}
where $B_1=0$ and $B_k=\sum_{j=2}^k(\theta_j-\theta_{j-1})\tau_{j-1}^2/2$, for $k=2,\ldots,K$. 

We therefore define a semiparametric inverse generator as the induced survival function, which can be written as
\begin{equation}
\label{eq:phiinv}
\varphi^{-1}(t)=\exp\{-H(t)\},
\end{equation}
where $H(t)$ is given in  \eqref{eq:H}. We now study some properties of this inverse generator. 

\begin{proposition}
\label{prop:1}
Consider the semiparametric inverse generator $\varphi^{-1}(t)$, given in \eqref{eq:phiinv}, and assume that $\{\theta_k, k=0,1,\ldots,K\}$ are such that $\theta_0>0$, $\theta_K=0$ and satisfy conditions (C1) and (C2) given by
\begin{enumerate}
\item[(C1)] $A_k+\theta_k t > 0$, for $t\in(\tau_{k-1},\tau_k]$ and for all $k=1,\ldots,K$.
\item[(C2)] $(A_k + \theta_k t)^2 > \theta_k$, for $t\in(\tau_{k-1},\tau_k]$ and for all $k=1,\ldots,K$.
\end{enumerate}
Then,
\begin{enumerate}
\item[(i)] $\varphi^{-1}(t)$ is a continuous and injective function of $t$, 
\item[(ii)] $\varphi^{-1}(t)$ is a convex function, 
\item[(iii)] $\varphi^{-1}(t)$ induces a strict generator. 
\end{enumerate}
\end{proposition}

\begin{proof}
For (i) we know that $h'(t)$, as in \eqref{eq:hp}, is a piecewise constant discontinuous function, however, function $h(t)$, as in \eqref{eq:h}, is continuous. To see this, for each $k=1,\ldots,K$, the limit from the left is $\lim_{t\to\tau_k^{-}} h(t) = \lim_{t \to \tau_k^{-}} A_k + \theta_k t = A_k + \theta_k \tau_k$, and the limit from the right becomes $\lim_{t \to \tau_k^{+}} h(t) = \lim_{t \to \tau_k^{+}} A_{k+1} + \theta_{k+1} t = A_{k+1} + \theta_{k+1} \tau_k$. Since $A_{k+1} = A_{k} + (\theta_{k} - \theta_{k+1})\tau_{k}$, then both limits coincide. For the second part of (i), we have that  $H(t)$ is a monotonous function whose derivative is strictly positive, due to condition (C1), therefore $H(t)$ is injective and invertible on its image \citep{rudin:87}. For (ii) we take the second derivative of $\varphi^{-1}(t)$ which becomes $\varphi^{-1(\prime \prime)}(t) =  \{h(t)\}^2 \exp\{-H(t)\} - h^{\prime}(t)\exp\{-H(t)\}$, this is  positive if and only if $\{h(t)\}^2-h'(t)>0$. For this to happen we require condition $(C2)$.  For (iii), $\varphi^{-1}(t)$ must be a proper survival function, that is, $h(t)$ must be nonnegative, which is achieved by imposing condition $(C1)$. Furthermore, we need $\lim_{t\to\infty}\varphi^{-1}(t)=0$, which is equivalent to prove that  $\lim_{t\to\infty}H(t)=\lim_{t\to\infty}\left(B_K+A_K t+\theta_K t^2/2\right)=\infty$. This is true since $B_K$ is a finite constant, by definition $\theta_K=0$, and this together with $(C2)$ imply $A_K>0$, so the linear part goes to infinity when $t\to\infty$. 
\end{proof}

By property (i) in Proposition \ref{prop:1}, we can invert equation \eqref{eq:phiinv} to obtain an expression for the generator. This is given by
\begin{align}
\nonumber
\varphi(t)=\sum_{k=1}^{K}&\left(\left[\sgn(\theta_k)\left\{\frac{2}{\theta_k}\left(\frac{A_k^2}{2\theta_k}-B_k-\log(t)\right)\right\}^{1/2}-\frac{A_k}{\theta_k}\right]I(\theta_k\neq 0)\right. \\
\label{eq:phi}
&\left.\hspace{5mm}-\frac{B_k+\log(t)}{A_k}I(\theta_k=0)\right)
I\left(\varphi^{-1}(\tau_k)\leq t< \varphi^{-1}(\tau_{k-1})\right).
\end{align}

The value $K$ controls the flexibility of the generator, and thus of the copula. If $K=1$, the induced Archimedean copula is the independence copula, whereas for larger $K$, the generator, and the induced copula, become semiparametric. Potentially $K$ could be infinite implying a nonparametric model. We now discuss some association properties of our semiparametric generator. 

To see the kind of association induced by our proposal, we computed the Kendall's tau using expression \eqref{eq:ktau} with generator \eqref{eq:phi}. This is given in the following result. 
\begin{proposition}
\label{prop:kt}
The Kendall's tau obtained by the Archimedean copula with semiparametric generator \eqref{eq:phi} is given by 
$$\kappa_\tau = -1+2\sum_{k=1}^K A_k \int_{\tau_{k-1}}^{\tau_k}\exp\left(-2B_k-2A_k t-\theta_k t^2\right)\d t.$$
Moreover, this $\kappa_\tau\in(-1,1)$.
\end{proposition}
\begin{proof}
Rewriting expression \eqref{eq:ktau} in terms of the inversed generator we obtain $\kappa_\tau=1-4\int_0^\infty t\{\varphi^{-1(\prime)}(t)\}^2\d t$. Computing the derivative we get $\varphi^{-1(\prime)}(t)=-\sum_{k=1}^K(A_k+\theta_k t)\times$ $\exp\{-(B_k + A_k t+\theta_k t/2)\}I(\tau_{k-1}<t\leq\tau_k)$. Doing the integral we obtain the expression. To obtain the range of possible values of $\kappa_\tau$ it is easier to re-write $\kappa_\tau$  in terms of $h(t)$ and $H(t)$. This becomes $\kappa_{\tau} = -2 \int_0^{\infty} t h^{\prime}(t)\exp\{-H(t)\}\,dt$. Here it is straightforward to see that the sign of $\kappa_\tau$ is determined by the sign of $h'(t)$, therefore $h'(t)>0$ for all $t$ implies $-1 < \kappa_{\tau} < 0$ and $h'(t)\le 0$ implies $0 \le\kappa_{\tau}<1$. 
\end{proof}

The expression for $\kappa_\tau$ tells us that the concordance induced by our semiparametric copula is a function of both, the parameters $\{\theta_k\}$, as well as of the partition limits $\{\tau_k\}$. It depends on a definite integral and can be evaluated numerically. What is more important is that $\kappa_\tau$ covers the whole range from $-1$ to $1$, showing that our proposal is very flexible. 

To illustrate the flexibility of our model we define a partition of the positive real line of size $K=10$, such that $\tau_k=-\log(1-k/10)$ for $k=0,1,\ldots,10$. We consider two scenarios for the values of the parameters $\{\theta_k\}$. The first scenario is defined by $\theta_k<0$ for all $k<K$, whereas the second scenario contains $\theta_k>0$ for all $k<K$. Conditions $(C1)$ and $(C2)$ were satisfied in both cases. Figure \ref{fig:ilust1} contains functions $h'(t)$, $H(t)$ and $\varphi^{-1}(t)$ for two different scenarios, the solid (blue) line corresponds to the first scenario and the dotted (red) line to the second scenario. In the first case the corresponding hazard function (middle panel) is decreasing, whereas for the second case the hazard function is increasing. The induced concordance values are $\kappa_\tau=0.368$ and $\kappa_\tau=-0.202$, respectively. 

As a second example, we consider a partition of size $K=50$, such that $\tau_k = -\log(1 - k/50)$ for $k=0,1,\ldots,50$. We consider three different scenarios for the parameters $\{\theta^{(i)}_k\}$ with $i=1,2,3$, respectively. In the first scenario we assume $\theta_1^{(1)}\sim \un(-1,1)$, in the second $\theta_1^{(2)}\sim \un(-50,0)$ and in the third $\theta_1^{(3)}\sim \un(0,1)$. Posteriorly, we define sequentially $\theta_k^{(i)}\sim \un(a_k^{(i)},b_k^{(i)})$ with $a_k^{(i)}$ and $b_k^{(i)}$ constants such that constraints $(C1)$ and $(C2)$ are satisfied, for $k=2,\ldots,K-1$ and $i=1,2,3$. We repeated sampling from these distributions a total of 5,000 times, and for each repetition we computed $\kappa_\tau$. The induced histogram densities  for the three scenarios are presented in Figure \ref{fig:simkt}. For the first scenario, the values of $\kappa_\tau$ range from $-0.3$ to $0.4$, showing that our model can capture both negative and positive concordance measures. For the second scenario, the values of $\kappa_\tau$ are all positive and the distribution is right skewed, and for the third scenario the values of $\kappa_\tau$ are all negative showing a left skewed distribution. 

According to \cite{nelsen:06}, new generators can be defined if we apply a scale transformation of the form $\phi^{-1}(t)=\varphi^{-1}(\beta t)$ if and only if $\phi(t)=\varphi(t)/\beta$, for $\beta>0$, where $\phi(t)$ becomes a new Archimedean copula generator. More recently, \cite{dibern&rulli:13} realised that the new generator $\phi(t)$ induces exactly the same copula \eqref{eq:arqc} as that obtained with $\varphi(t)$. To see this we have that $C_\phi(u,v)=\phi^{-1}(\phi(u)+\phi(v))=\varphi^{-1}\left(\beta\left\{\frac{1}{\beta}\varphi(u)+\frac{1}{\beta}\varphi(v)\right\}\right)=C_\varphi(u,v)$. In other words, an Archimedean copula generator is not unique. 

Moreover, in terms of the hazard rate functions, $h_\phi(t)$ and $h_\varphi(t)$, induced by generators $\phi$ and $\varphi$, respectively, the relationship becomes $h_\phi(t)=\beta h_\varphi(\beta t)$. In order to make our semiparametric generator identifiable, without loss of generality, we impose the new constraint  
\begin{enumerate}
\item[(C3)] $\theta_0=1$.
\end{enumerate}
This constraint is equivalent to imposing $h(0)=1$ in definition \eqref{eq:h}. 

\subsection{Alternative construction}
\label{subsec:alternative}

Instead of starting with a piecewise function for the derivative of a hazard rate, we could start by defining a piecewise constant function for the hazard rate itself. That is $h(t)=\sum_{k=1}^K\theta_k I(\tau_{k-1}<t\leq\tau_k)$ with $\theta_k>0$, and $\{\tau_k\}$ a partition of the positive real line. In this case the cumulative hazard function becomes $H(t)=\sum_{j=1}^{k}\theta_j\Delta_j+\theta_{k}(t-\tau_{k-1})$, for $t\in(\tau_{k-1},\tau_{k}]$, with $\Delta_j=\tau_{j}-\tau_{j-1}$. The inverse generator is then a linear spline of the form 
\begin{equation}
\nonumber
\varphi^{-1}(t)=\exp\left\{-\sum_{k=1}^K\theta_k w_k(t)\right\},
\end{equation}
with $$w_k(t)=\left\{\begin{array}{ll}
\Delta_k, & t>\tau_k \\
t-\tau_{k-1} & t\in(\tau_{k-1},\tau_k] \\
0 & \mbox{otherwise} 
\end{array}\right.$$
and the corresponding Archimedean generator has the form
\begin{equation}
\nonumber
\varphi(t)=\sum_{k=1}^K \left\{\tau_{k-1}-\frac{1}{\theta_{k}}(\log t+\vartheta_{k-1})\right\}I(\vartheta_{k-1}<-\log t\leq\vartheta_k),
\end{equation}
with $\vartheta_k=\sum_{j=1}^{k}\theta_j\Delta_j$. To ensure convexity of the generator we further require $\theta_1\geq\theta_2\geq\cdots\geq\theta_K$. Furthermore, the Kendall's tau has a simpler expression
\begin{equation}
\nonumber
\kappa_\tau=1+\sum_{k=1}^K\left\{e^{-2\vartheta_k}(1+2\theta_k\tau_k)-e^{-2\vartheta_{k-1}}(1+2\theta_k\tau_{k-1})\right\}.
\end{equation}

However, it can be shown that this expression for the Kendall's tau only allows positive values, constraining the possible associations captured by the model. Therefore, in the remainder of the paper we will concentrate on our main proposal defined in Section \ref{subsec:main}.

\section{Posterior inference}
\label{sec:post}

The copula density $f_C(u,v)$, of an Archimedean copula, can be obtained by taking the second crossed derivatives with respect to $u$ and $v$ in expression \eqref{eq:arqc}. In terms of the generator and its inverse this density becomes
\begin{equation}
\label{eq:cdensity}
f_C(u,v)=\varphi^{-1(\prime\prime)}\left(\varphi(u)+\varphi(v)\right)
\varphi^{(\prime)}(u)\varphi^{(\prime)}(v),
\end{equation}
where the single and double primes denote first and second derivatives, respectively, and are given by
$$\varphi^{-1(\prime\prime)}(t)=\sum_{k=1}^K \left\{(A_k+\theta_k t)^2-\theta_k\right\}\exp\left\{-\left(B_k+A_k t+\frac{\theta_k}{2}t^2\right)\right\}I(\tau_{k-1}<t\leq\tau_K)$$
and
$$\varphi^{(\prime)}(t)=-\sum_{k=1}^{K}\frac{1}{t}\left(-2\theta_k B_k+A_k^2-2\theta_k\log(t)\right)^{-1/2}I\left(\varphi^{-1}(\tau_k)\leq t< \varphi^{-1}(\tau_{k-1})\right).$$

Let $(U_i,V_i)$, $i=1,\ldots,n$ be a bivariate sample of size $n$ from $f_C(u,v)$ defined in \eqref{eq:cdensity}. With this we can construct the likelihood for $\btheta=(\theta_0,\theta_1,\ldots,\theta_K)$ as $\mbox{lik}(\btheta\mid\bu,\bv)=\prod_{i=1}^n f_C(u_i,v_i\mid\btheta)$, where we have made explicit the dependence on $\btheta$ in the notation of the copula density. Recall that the parameter space $\Theta$ contains the values of $\btheta$ that satisfy several conditions, $(C1)$ and $(C2)$ given in Proposition \ref{prop:1}, $(C3)$ to make our generator unique, and $\theta_K=0$. 

We assume a joint prior distribution for $\btheta$ of the form 
\begin{equation}
\label{eq:prior}
f(\btheta)\propto\prod_{k=1}^{K-1} \left\{\pi_0 I(\theta_k=0)+(1-\pi_0)\no(\theta_k\mid \mu_0,\sigma^2_0)\right\} I(\btheta\in\Theta). 
\end{equation}

Note that we explicitly allow the $\theta_k$'s, for $k=1,\ldots,K-1$ to be zero with positive probability $\pi_0$. This prior choice is useful to define an independence test. Specifically, we consider the hypotheses $H_0:U \mbox{ and } V$ independent, which is equivalent to $H_0:\theta_1=\cdots=\theta_{K-1}=0$, versus the alternative $H_1:U \mbox{ and } V$ dependent, which is equivalent to $H_1:\theta_k\neq 0$ for at least one $k=1,\ldots,K-1$. To perform the test we can compute the posterior probabilities of $H_0$ and $H_1$ and make the decision using decision theory \citep{degroot:04}, or use the corresponding Bayes factor \citep{kass&raftery:95}, which in case that $\P(H_0)=\P(H_1)$, this becomes the odds in favour of $H_1$, that is, $B_{10}=\P(H_1\mid\data)/\P(H_0\mid\data)$. Here we follow the approach of \cite{filippi&al:16} and report $\P(H_1\mid\data)$ as an evidence in favour of dependence.

The posterior distribution of $\btheta$ is simply given by the product of expressions \eqref{eq:cdensity} and \eqref{eq:prior}, up to a proportionality constant. It is easier to characterize the posterior distribution by implementing a Gibbs sampler \citep{smith&roberts:93} and sampling from the conditional posterior distributions 
\begin{equation}
\label{eq:postc}
f(\theta_k\mid \btheta_{-k},\data)\propto\mbox{lik}(\btheta\mid\bu,\bv)f(\btheta),
\end{equation}
for $k=1,\ldots,K-1$. However, sampling from conditional distributions \eqref{eq:postc} is not trivial since the parameter $\theta_k$ appears everywhere in the likelihood, the parameter space is complex and no closed expression can be obtained for the normalising constant, we therefore propose a Metropolis-Hastings step \citep{tierney:94} by sampling $\theta_k^*$ at iteration $(r+1)$ from a random walk proposal distribution 
$$q(\theta_k\mid\btheta_{-k},\theta_k^{(r)})=\pi_1 I(\theta_k=0)+(1-\pi_1)\un(\theta_k\mid \max\{a_k,\theta_k^{(r)}-\delta c_k\},\min\{b_k,\theta_k^{(r)}+\delta c_k\})$$
where the interval $(a_k,b_k)$ represents the conditional support of $\theta_k$, $c_k=b_k-a_k$ is its length, with $a_k = \max_{k \le j \le K-1} \left\{\left(\sqrt{\theta_{j+1}}I(\theta_{j+1}\geq 0) -\theta_0 - \sum_{i=1,i \ne k}^{j}(\tau_i - \tau_{i-1})\theta_i\right)/(\tau_{k} - \tau_{k-1})\right\}$, for $k=1,\ldots,K-1$, $b_k = \left(\theta_0 + \sum_{j=1}^{k-1} (\tau_j - \tau_{j-1})\theta_j \right)^2$, for $k=2,\ldots,K-1$, and $b_1 = 1$. The justification of these bounds obeys the inclusion of constraints $(C1)$ and $(C2)$ and their derivations are given in Appendix \ref{sec:appendix}. The parameters $\pi_1$ and $\delta$ are tuning parameters that control the acceptance rate. 

Therefore, at iteration $r+1$ we accept $\theta_k^*$ with probability 
$$p\left(\theta_k^*,\theta_k^{(r)}\right)=\min\left\{1\,,\;\frac{f(\theta_k^*\mid\btheta_{-k},\data)\,q(\theta_k^{(r)}\mid\btheta_{-k},\theta_k^{*})}{f(\theta_k^{(r)}\mid\btheta_{-k},\data)\,q(\theta_k^*\mid\btheta_{-k},\theta_k^{(r)})}\right\}.$$
This Metropolis-Hastings within Gibbs procedure to obtain posterior inference of our model was implemented in Python and the code is available upon request from the first author. 

To perform the independent test, posterior probability of $H_0$ can be approximated via Monte Carlo by using the MCMC posterior draws of the vector $\btheta$ and computing the relative frequency of the event $\theta_k=0$ for all $k=1,\ldots,K-1$, we therefore obtain posterior probability of $H_1$ by computing the complement.

\section{Numerical studies}
\label{sec:illust}

We illustrate the performance of our model in two ways, through a simulation study, and with a real data set. 

To define the partition $\{\tau_k\}$ of the positive real line, inspired by the generator of the product copula, we consider a Log-$\alpha$ partition defined by $\tau_k = - \alpha \log(1-k/K)$ for $k = 0,\ldots,K-1$, with $\alpha>0$. This partition is the result of transforming a uniform partition in the interval $[0,1]$ via a convex function. Larger values of $\alpha$ increase the spread of the partition along the positive real line. 

\subsection{Simulation study}

We generated simulated data from four parametric Archimedean copulas, namely the product, Clayton, Ali-Mikhail-Haq (AMH) and Gumbel copulas. Their features are summarised in Table \ref{tab:parcopulas}, where we include the parameter space, the generator, the inverse generator, an indicator whether the copula is strict or not and the induced $h(t)$ function obtained through inversion of relationship \eqref{eq:phiinv}.  

Note that, due to the nonunicity of an Archimedean generator, an equivalent constraint to $(C3)$ has to be imposed to the parametric generators that we are going to compare to. That is we set $h(0)=1$ for the product, Clayton and AMH copulas, and $h(\epsilon)=1$ for the Gumbel copula, for say $\epsilon=0.01$. The difference in the latter case is because, for a Gumbel copula, $h(t)\to\infty$ when $t\to 0$. These conditions are already included in the parametrisation used in Table \ref{tab:parcopulas}. 

For each parametric copula we took a sample of size $n=200$. To specify the copulas we took particular values in the parametric space that induce negative and positive dependence. In particular we set  $\theta\in\{-0.4,-0.8,0.6,1\}$ for the Clayton copula, $\theta\in\{-0.3,-0.7,0.3,0.7\}$ for the AMH copula, and $\theta\in\{1.4,2.0\}$ for the Gumbel copula. For the partition size we compared $K\in\{10,20\}$ and tried values $\alpha \in \{0.3,0.5,0.9,1,2,\ldots,10\}$.

For the prior distributions \eqref{eq:prior} we took $\pi_0=0, \mu_0=-1$ and $\sigma_0^2=10$. We implemented a MH step within the Gibbs sampler where the proposal distributions were specified by $\pi_1=0$ and $\delta=0.25$. The acceptance rate attained with these specifications are around 30\%, which according to \cite{robert&casella:10} are optimal for random walks. Finally, the chains were ran for 20,000 iterations with a burn-in of 2,000 and keeping one of every 5$^{th}$ iteration to produce posterior estimates. Convergence of the chains was assessed informally by looking at the trace and ergodic means plots. Computational times using an \emph{intel core i7 microprocessor} averaged 6 and 15 minutes for the partition sizes $K=10$ and $K=20$, respectively.

To assess goodness of fit (GOF) we computed several statistics. The logarithm of the pseudo marginal likelihood (LPML), originally suggested by \cite{geisser&eddy:79}, to assess the fitting of the model to the data. The supremum norm, defined by $\sup_{t}|\varphi^{-1}(t)-\widehat{\varphi}^{-1}(t)|$ to assess the discrepancy between our posterior estimate (posterior mean) $\widehat{\varphi}^{-1}(t)$ from the true inverse generator $\varphi^{-1}(t)$. We also computed the Kendall's tau coefficient and compare the point (posterior mean) and 95\% interval estimates with the true value. Additionally, as a graphical aid to see the performance of our model, we compare the posterior estimates, point (posterior mean) and 95\% pointwise credible intervals, of functions $h(t)$ and $\varphi^{-1}(t)$ with the true ones. In general, the idea of our model is to properly estimate the joint density of a particular dataset, say $f(u,v)$, but in Archimedean copulas such a density is characterised by the generator, like in \eqref{eq:cdensity}. This is why we concentrate on comparing the inverse generator and its associated hazard function. 

To avoid overwhelming the reader with many tables and graphs, we only show results for some of the simulated datasets to illustrate, the performance of our model in the other datasets not shown is analogous. The GOF statistics are shown in Tables \ref{tab:prod} to \ref{tab:gumbel14}. Although we fitted our model with all values of $\alpha$ mentioned above, we only show results for those around the best fitting model in the tables. Posterior estimates of the functions are depicted in Figures \ref{fig:prod} to \ref{fig:gumbel14}. Here we only show estimates with the best fitting model. 

For the product copula the GOF measures are presented in Table \ref{tab:prod}. With exception of the partition Log-$3$ for $K=20$, for all settings considered, the true $\kappa_\tau$ lies inside the 95\% credible intervals. The LPML chooses the model with Log-$1$ partitions of size $K=10$, and corresponds to the second smallest value of the supremum norm. Posterior estimates of functions $h(t)$ and $\varphi^{-1}(t)$ are shown in Figure \ref{fig:prod}. In both cases the true function lies inside the 95\% credible intervals. 

For the Clayton copula we have two choices of $\theta$, $-0.8$ and $1$. The first choice, $\theta=-0.8$, corresponds to a generator that is not strict, that is, $\varphi^{-1}(t)>0$ for $t\in[0,5/4]$, and $\varphi^{-1}(t)=0$ for $t>5/4$. This is an interesting challenge because our model defined only strict generators. The settings with smallest supremum norm, Log-$0.5$ with $K=10$, produces the 95\% credible interval for $\kappa_\tau$ closest to the true value, however it does not achieve the largest LPML. The inconsistency of the GOF measures might be due to
the non strictness feature of the true generator. Moreover, if we look at the graphs of the posterior estimates of $h(t)$ and $\varphi^{-1}(t)$ (Figure \ref{fig:clay_08}), for larger values of $t$ the true functions lie outside of our posterior estimates. For $\theta=1$, the best model is obtained with a Log-$6$ partition of size $10$. In this case, posterior estimates of functions $h(t)$ and $\varphi^{-1}(t)$ with the best fitting (Figure \ref{fig:clay1}), contain the true functions. 

For the AMH copula we have two values of $\theta$, $-0.7$ and $0.7$. The best fitting chosen by at least two of the three GOF criteria is obtained with a Log-$6$ and Log-$1$ partitions of size $K=10$, respectively for the two values of $\theta$. Posterior estimates of functions $h(t)$ and $\varphi^{-1}(t)$ with the best fitting are shown in Figures \ref{fig:amh_07} and \ref{fig:amh07}, respectively. In all cases the true functions lie within the 95\% credible intervals.

For the Gumbel copula with $\theta=1.4$ we have an interesting behaviour. The true $h(t)$ function has the feature that $h(0)=\infty$. This represents a challenge for our model since we have imposed the constraint $(C3)$ which is equivalent to $h(0)=1$. The highest LPML value is obtained with a Log-$7$ partition of size $K=10$, however the posterior 95\% credible interval for $\kappa_\tau$ does not contain the true value. On the other hand, the second best value of LPML is obtained with an Log-$3$ partition of  size $10$, and in this case the 95\% credible interval for $\kappa_\tau$ does contain the true value. We select this latter as the best fitting. Posterior estimates of functions $h(t)$ and $\varphi^{-1}(t)$ are shown in Figure \ref{fig:gumbel14}. Recalling that the true hazard function goes asymptotically to infinity when $t\to 0$, therefore, for values close to zero the true $h(t)$ lies outside our posterior credible intervals, something similar happens in the estimates of the inverse generator. Apart from this, our posterior estimates are very good for $t>\epsilon$. 

An important learning from the previous examples is that increasing the partition size does not necessarily imply better fitting. 

\subsection{Real data analysis}

In public health it is important to study the factors that determine the birth weight of a child. Low birth weight is associated with high perinatal mortality and morbility \citep[e.g.][]{stevens&orleans:01}. 
We study the dependence structure between the age of a mother ($X$) and the weight of her child ($Y$), and concentrated on mothers of 35 years old and above. The dataset was obtained from the General Hospital of Mexico through the open data platform that can be accessed at \emph{https://datos.gob.mx/busca/dataset/perfiles-metabolicos-neonatales/resource/4ab603eb-b73a-498f-8c56-0dc6d21930e8}. It contains $n=208$ records from the first sample of the neonatal metabolic profile of male babies registered in the year 2017 in Mexico City. 

The marginal distributions for variables $U$ and $V$, induced by copula \eqref{eq:arqc}, are uniform. In practice, copulas are used to model the dependence for any pair of random variables regardless of their marginal distributions. Let $X$ and $Y$ be two random variables with marginal cumulative distributions $F(x)$ and $G(y)$ respectively. Then the joint cumulative distribution function for $(X,Y)$ is obtained as \citep{sklar:59}, $H(x,y)=C(F^{-1}(x),G^{-1}(y))$, where $C$ is given in \eqref{eq:arqc}. 

Since we are just interested in modelling the dependence between $X$ and $Y$, it is common in practice to transform the original data, $(X_i,Y_i)$, $i=1,\ldots,n$, to the unit interval via a modified rank transformation \citep{Deheuvels:79} in the following way. Let $\bX'=(X_1,\ldots,X_n)$ and $\bY'=(Y_1,\ldots,Y_n)$ then $U_i=\mbox{rank}(i,\bX)/n$ and $V_i=\mbox{rank}(i,\bY)/n$ are the transformed data, where $\mbox{rank}(i,\bX)=k$ if and only if $X_i=X_{(k)}$ for $i,k=1,\ldots,n$. This is based on the probability integral transform using the empirical cumulative distribution function of each coordinate.

In Figure \ref{fig:realdata} we show a dispersion diagram of the original data (left panel) and the rank transformed data (right panel). To avoid problems due to ties in the original data, we first include a perturbation to the data by adding a uniform random variable  $\un(-0.0,0.01)$ to each coordinate. The sample Kendall's tau value for the transformed data is $\tilde{\kappa}_{\tau} = -0.1162$.

We fitted our model to the transformed data with the following specifications. To define the partitions we took values $\alpha \in \{0.3,0.5,0.9,1,2,\ldots,10\}$ with sizes $K\in\{10,20\}$.
For the prior we took $\pi_0=0$, $\mu_0=-1$ and $\sigma_0^2=10$. The MCMC specifications were the same as those used for the simulated data. 

The GOF measures computed were the LPML and the posterior estimates (point and 95\% credible interval) of $\kappa_\tau$. The results are reported in Table \ref{tab:realdata}. The best fitting model according to LPML is that obtained with a partition of size $K=10$ and Log-$10$. The sample concordance $\tilde{\kappa}_\tau$ is included in our posterior 95\% credible interval estimate $\kappa_\tau\in(-0.213,-0.098)$. 

The estimated hazard rate function $h(t)$ and the inverse generator $\varphi^{-1}(t)$, with the best fitting model, are included in the top row in Figure \ref{fig:perealdata}. The solid thick line corresponds to the point estimates and the solid thin lines to the 95\% credible intervals. For a visual comparison, the blue dotted line corresponds to the functions of the independence (product) copula. Additionally, we include an estimate of the joint density as well as the corresponding contour plots (bottom row in Figure \ref{fig:realdata}. These estimates suggest that there is a negative (weak) dependence between the age of the mother and the birth weight of the child. The older the mother, the less weight of the child. This finding could potentially help the policy makers to focus campaigns to help the awareness of future mothers.

\subsection{Independence test}

As mentioned in Section \ref{sec:model}, we can use our model to undertake an independence test. For that we choose the prior distribution for the $\theta_k$'s, as in \eqref{eq:prior}, such that the prior probability of  $H_0:\theta_1=\cdots=\theta_{K-1}=0$ is $1/2$, in other words, we want $\P(H_0)=\pi_0^{K-1}=1/2$. Particularly, for a partition of size $K=10$ we need to specify $\pi_0=0.9258$. In order to get a point of mass proposal in the MH step we consider $\pi_1 = 0.3$. We re-ran our model using these values with the other specifications left unchanged and performed the test for all simulated and real datasets. 

To place our test in context, we compare our results with the recently proposed independence test of \cite{filippi&al:16}, based on Dirichlet process mixture models. These authors actually proposed two tests, one based on a contingency table approach (CT) and another based on a mixture model approach (MM). Additionally, we implemented a frequentist test based on the empirical copula (EC) given in \cite{Deheuvels:79}. For the three Bayesian tests, ours (SPAC) and the other two competitors, we report $\P(H_1\mid\data)$, whereas for the frequentist test we report the p-values. All these values are included in Table \ref{tab:indep}. 

We first mention that the values $\P(H_1\mid\data)$ from the Bayesian tests have to be calibrated with respect to that obtained for the product (independent) dataset. The three tests assign small evidence of dependence to the product dataset, as it should be, whereas the frequentist test assigns a p-value of 0.15 to the same product dataset, which is large enough to not to reject the null hypothesis of independence. 

For the Clayton and Gumbel datasets, all four test are consistent giving enough evidence to dependence. For the AMH datasets we have mixed decisions. None of the four tests are able to detect dependence for the cases of $\theta=-0.3$ and $\theta=0.3$. This is understandable since the AMH copula produces data that look similar to the product copula for values of $\theta$ close to zero. For the other two values, $\theta=-0.7$ and $\theta=0.7$, the frequentist test EC does not detect dependence, however the Bayesian tests give more evidence of dependence, being our SPAC test the one that best supports dependence for these two datasets. 

Finally, for the real dataset, we also have mixed decisions. Tests CT and EC do not detect any dependence, however, our new test SPAC and MM give enough support to dependence, which is also consistent to the estimated generator obtained with our model and presented in the right panel of Figure \ref{fig:perealdata}.

\section{Concluding remarks}
\label{sec:concl}

We have proposed a semiparametric Archimedean copula that is flexible enough to capture the behaviour of several families of parametric Archimedean copulas. Our model is capable of modelling positive and negative dependence. The number of parameters in the model to produce a good estimation of the dependence in the data should not be extremely high. For most of the examples considered here ten parameters are enough. 

Defining an appropriate partition to analyse real data sets is not trivial. We suggest to try different values of $\alpha$ in a wide range and compare using a GOF criteria like the LPML we used here. 

Our proposal is also suitable to perform an independent test, which compares favourable with alternative independence tests. For the datasets considered here, our proposal assigned the largest evidence of dependence for the dependent datasets. 

In the exposition and in examples considered here, we concentrated on bivariate copulas, however extensions to more than two dimensions is also possible, say $C(u_1,\ldots,u_m)=\varphi^{-1}\left(\varphi(u_1)+\cdots+\varphi(u_m)\right)$. Performance of our semiparametric copula in this multivariate setting is worth studying.

Our model is motivated by semiparametric proposals for survival analysis functions \citep{nieto&walker:02} and appropriately modified to satisfy the properties of an Archimedean generator. The semiparametric generator presented here turned out to be based on quadratic splines, however, alternative proposals are possible as the one described in Section \ref{subsec:alternative}. 

Although, the motivation of our proposal lies within a survival analysis context, the inclusion of right censored observations into the analysis is not straightforward. The likelihood contribution would involve the cumulative distribution function of the bivariate density induced by the copula, and this is not available in closed form. A data augmentation technique, like those in \cite{tanner:91}, would be the way to proceed.

\section*{Acknowledgements}
This research was done while the first author was doing a post doctorate at the Department of Statistics, ITAM, and was supported by \textit{Asociaci\'on Mexicana de Cultura, A.C.} The authors are grateful to three anonymous referees for their insightful comments.

\appendix

\section*{Appendix}
\label{sec:appendix}

\textbf{Derivation of posterior conditional support of $\theta_k$.}

\noindent
In order to satisfy constraint $(C1)$, we consider first the case $\theta_k\leq 0$. Therefore \\
$\min_{t \in (\tau_{k-1},\tau_k]}A_k + \theta_k t = A_k + \theta_k \tau_{k}$. This implies the following constraint for $\theta_k$, 
$$\theta_k \geq \max_{k \le j \le K-1} \left\{- \left( \theta_0 + \sum_{i=1,i \ne k}^{j}(\tau_i - \tau_{i-1})\theta_i\right)/(\tau_{k} - \tau_{k-1})\right\},$$
for $k = 1,\ldots,K-1$, where we define the empty sum as zero. 

On the order hand, if $\theta_k > 0$ we have
$\min_{t \in (\tau_{k-1},\tau_k]}A_k + \theta_k t = A_k + \theta_k \tau_{k-1},$
and we get, from condition $(C2)$, the following restriction
$$\theta_k < \left(\theta_0 + \sum_{i=1}^{k-1} (\tau_i - \tau_{i-1})\theta_i \right)^2.$$
This defines the upper bound $b_k$, for $k = 2,\ldots,K-1$, and $b_1=1$. 

Because the term $\theta_k$ appears on the right side of the previous inequality for $j=k+1,\ldots,K-1$, we need to consider the following restriction
$$\theta_k > \left.\left( \sqrt{\theta_j} -  \theta_0 - \sum_{i=1,i \ne k}^{j-1}(\tau_i - \tau_{i-1})\theta_i\right)\right/(\tau_{k} - \tau_{k-1})$$
if $\theta_j \geq 0$. Combining this with the constraint when $\theta_k\leq 0$ above, we get the lower bound $a_k$ for $k=1,\ldots,K-1$.

\bibliographystyle{natbib}

\newpage

\begin{figure}[h]
\centerline{
\includegraphics[width=5.5cm,height=5.cm]{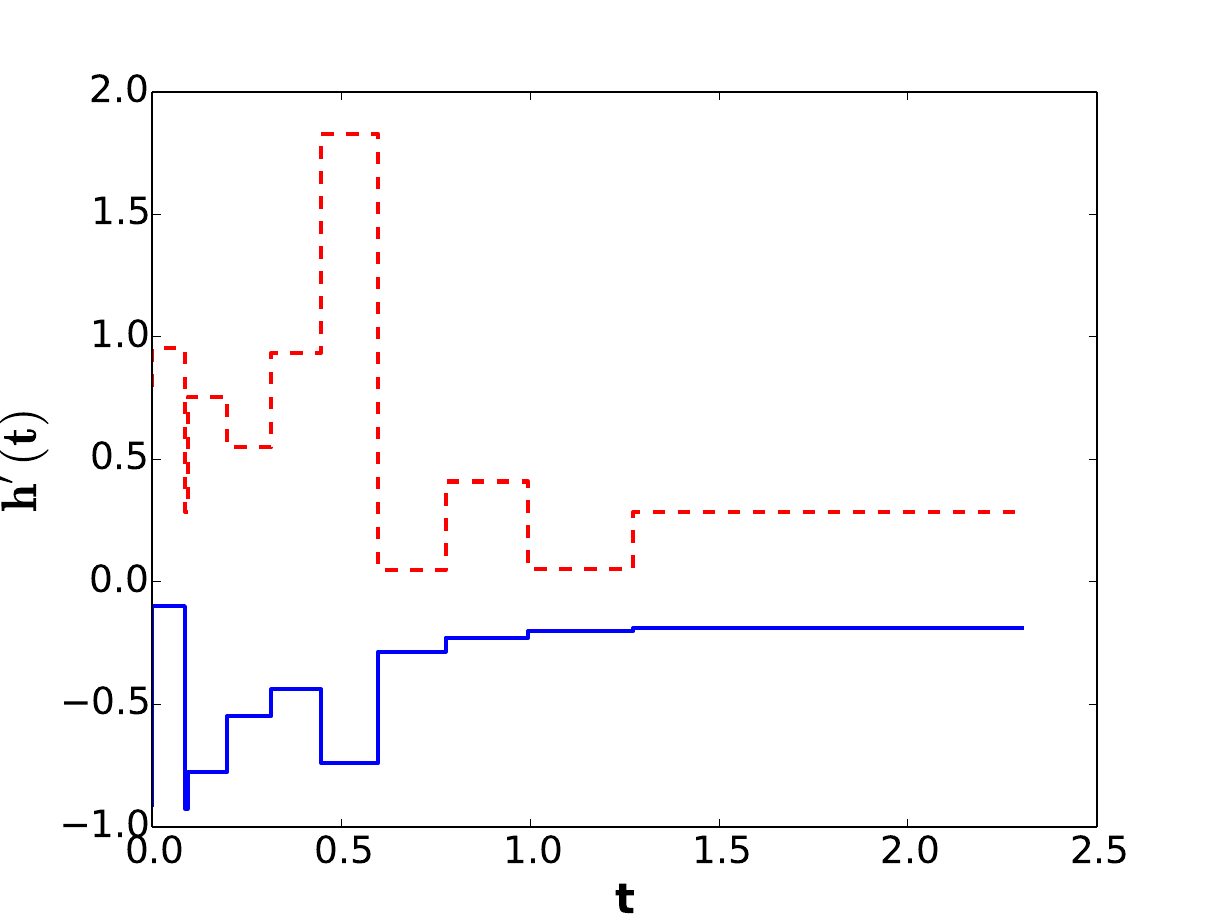}
\includegraphics[width=5.5cm,height=5.cm]{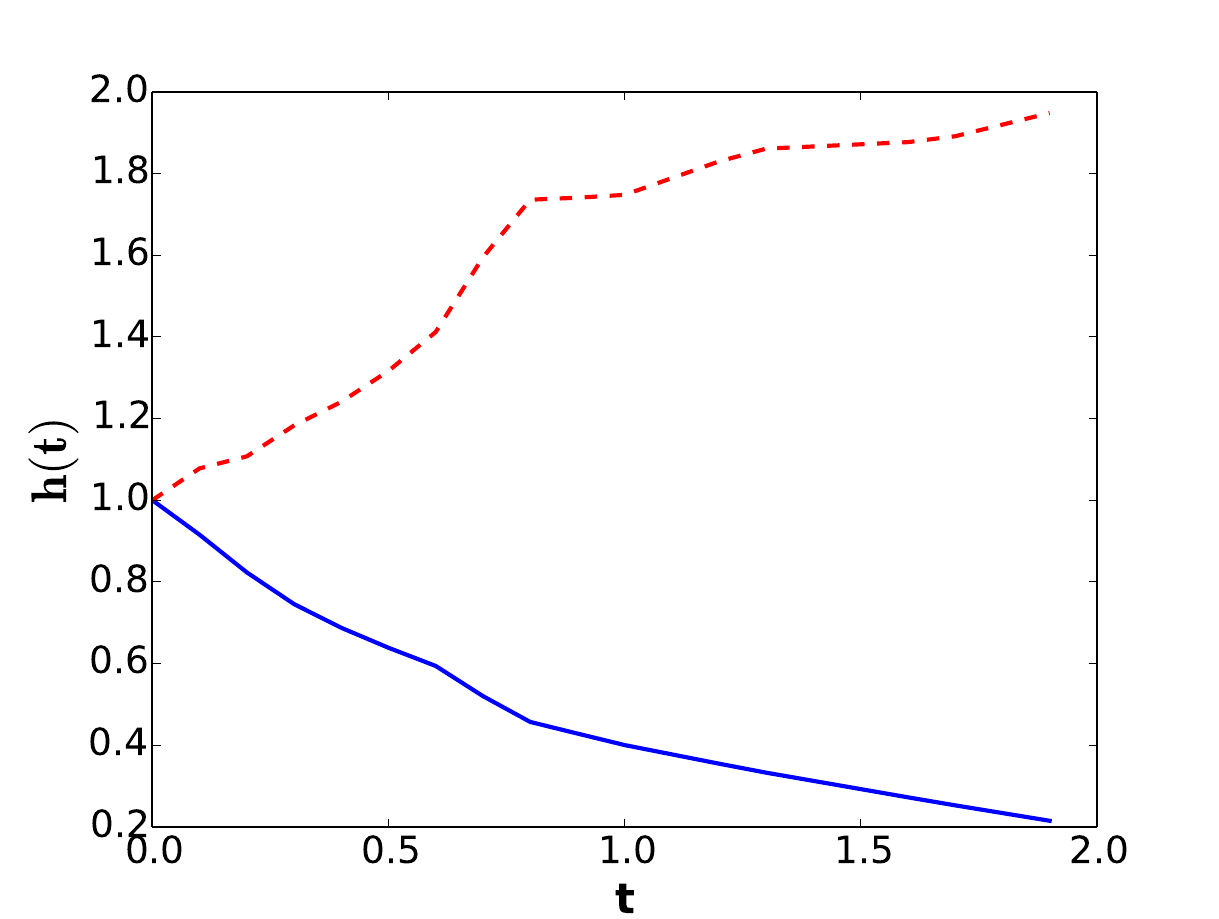}
\includegraphics[width=5.5cm,height=5.cm]{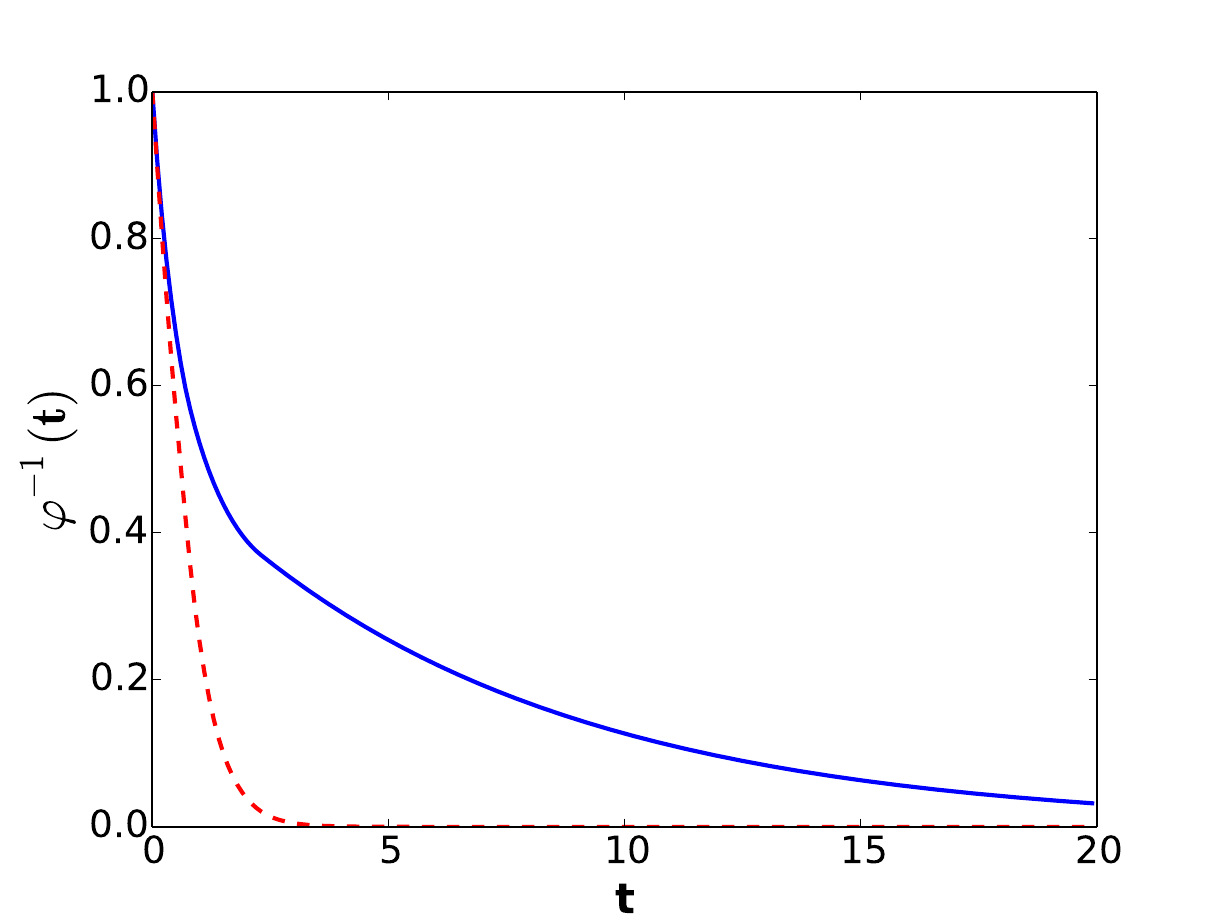}
}
\caption{Functions $h'(t)$ (first panel), $h(t)$ (second panel) and $\varphi^{-1}(t)$ (third panel) for two scenarios of $\{\theta_k\}$. All negative values (solid line), and all positive values (dotted line).}
\label{fig:ilust1}
\end{figure}

\begin{figure}[h]
\centerline{
\includegraphics[width=5.1cm,height=5.cm]{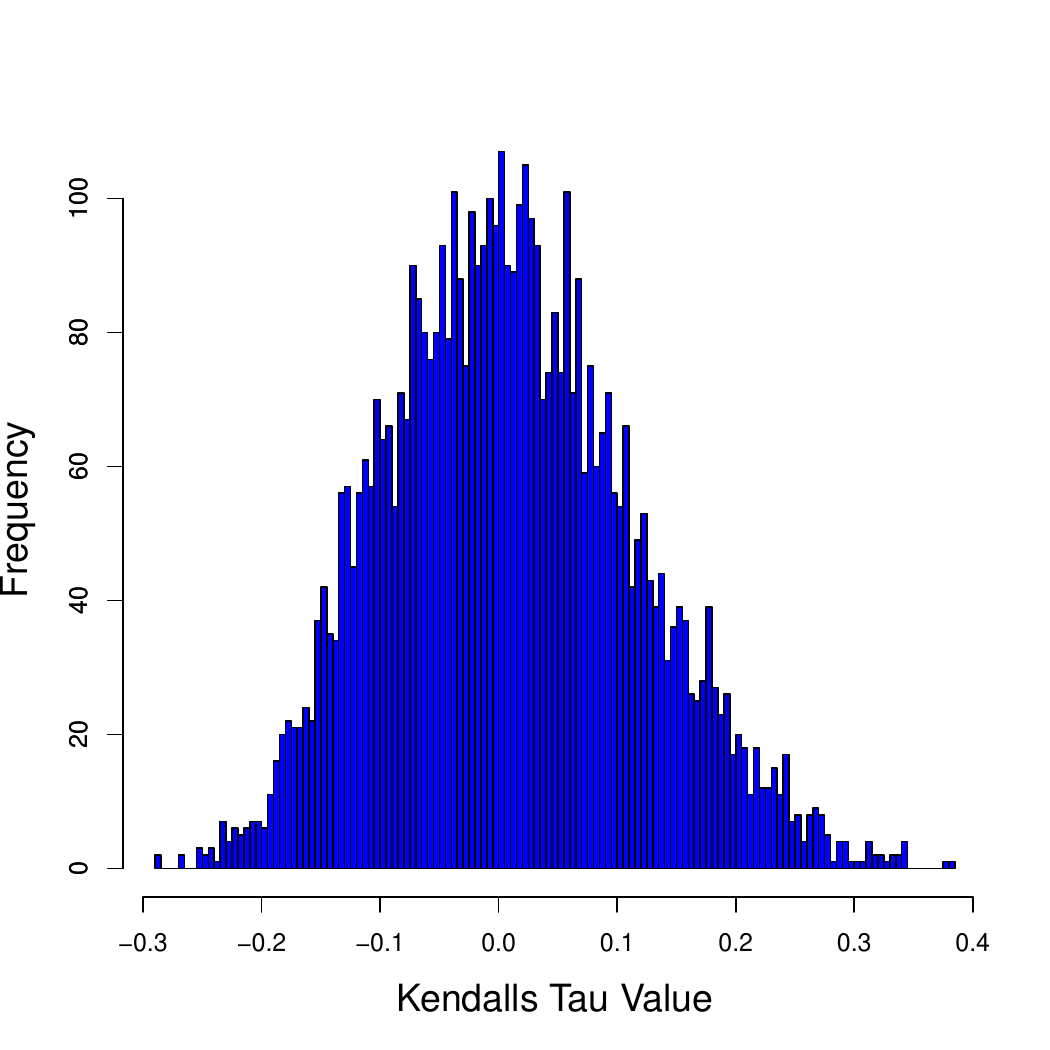}
\includegraphics[width=5.1cm,height=5.cm]{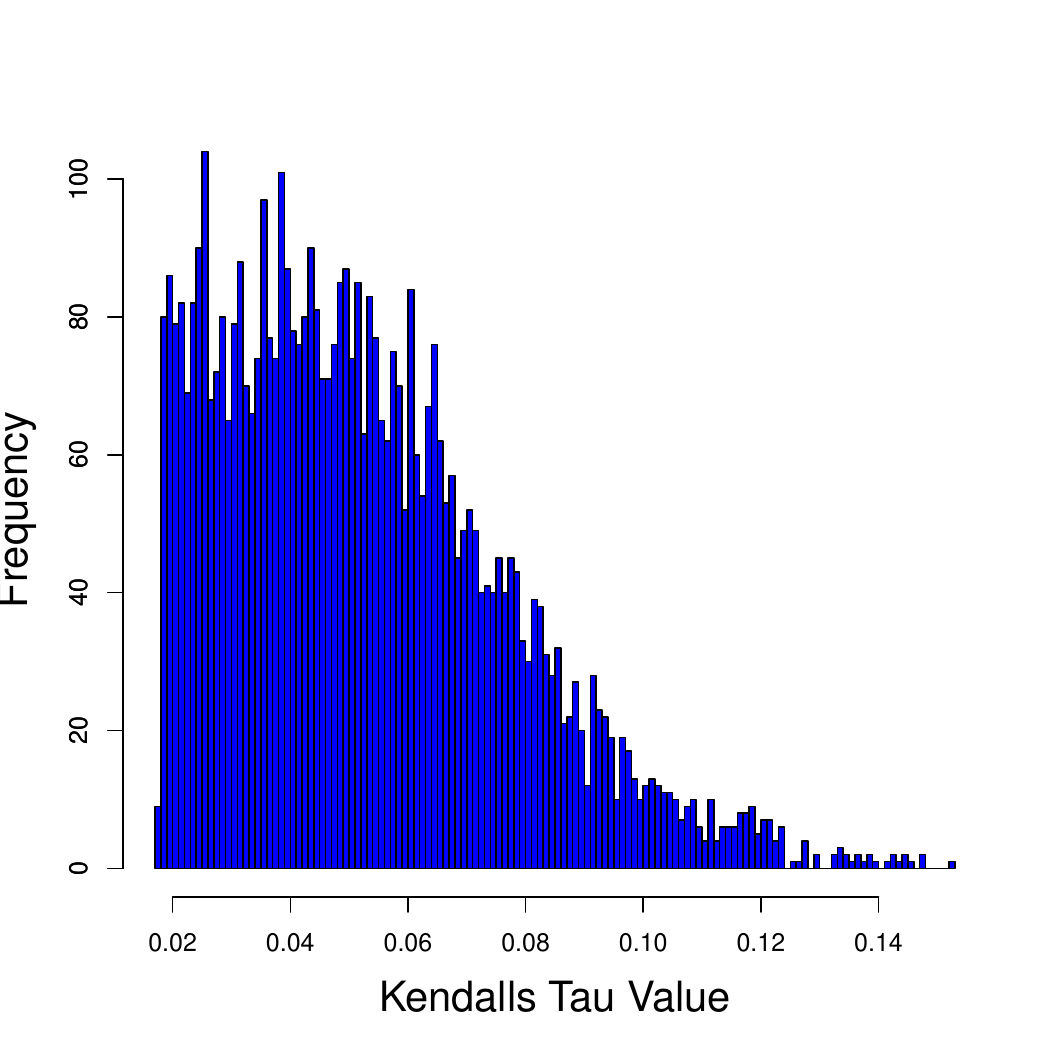}
\includegraphics[width=5.1cm,height=5.cm]{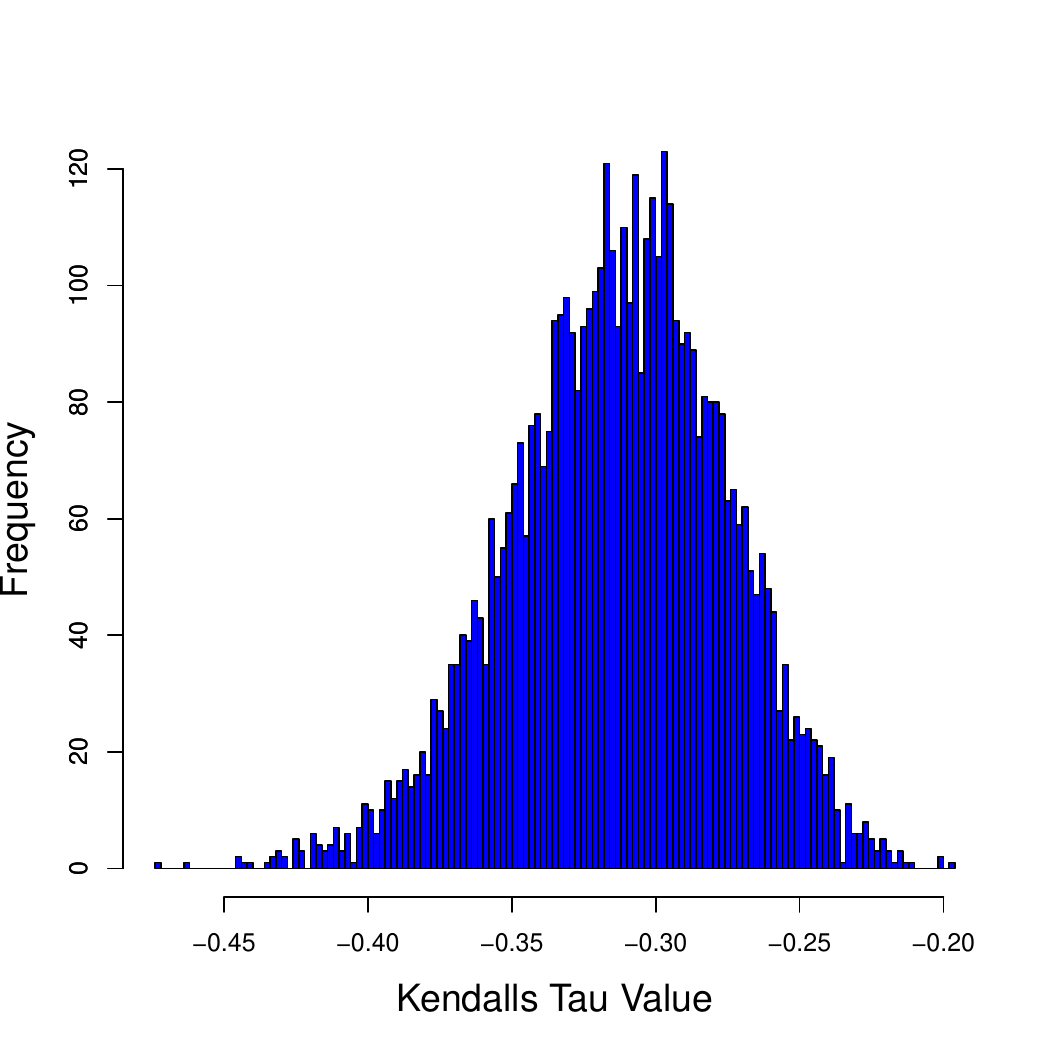}
}
\caption{Prior distributions of Kendall's tau, induced by our model, under three different scenarios.}
\label{fig:simkt}
\end{figure}

\begin{figure}[H]
\begin{center}
\includegraphics[width=6.5cm,height=6.0cm]{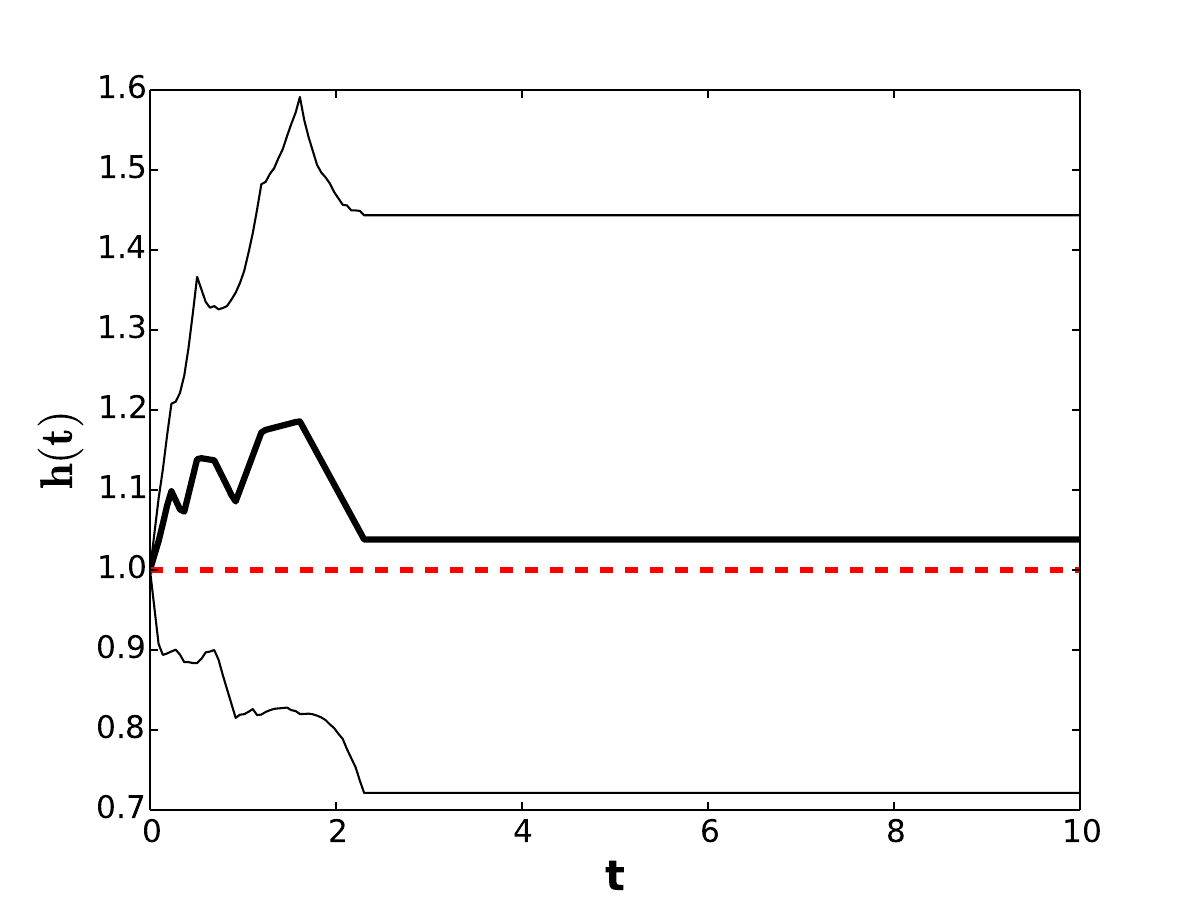}
\includegraphics[width=6.5cm,height=6.0cm]{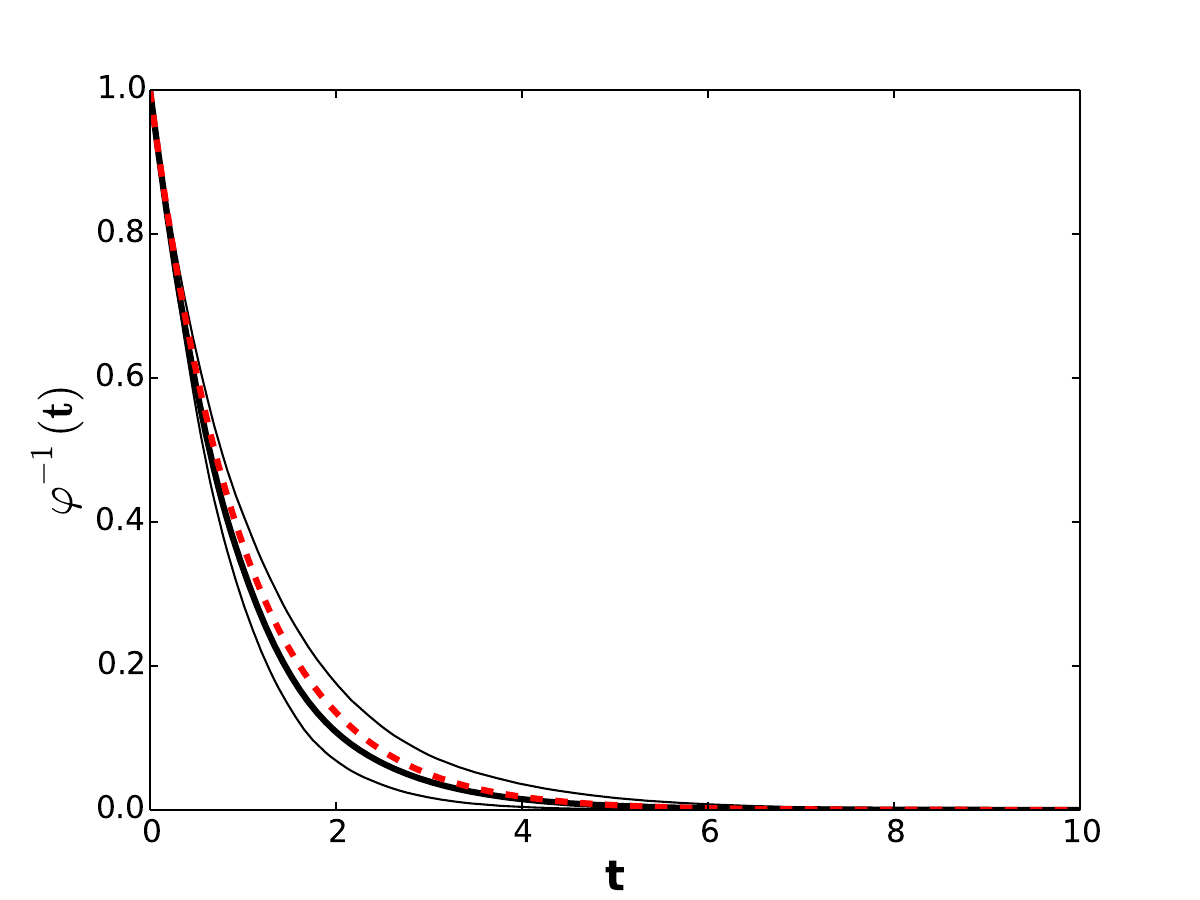}
\caption{\small{Posterior estimates of $h(t)$ and $\varphi^{-1}(t)$, obtained with a Log-$1$ partition of size $K=10$, for a simulated dataset of size $n=200$ from the product copula. Posterior mean (thick solid line), 95\% pointwise credible intervals (thin solid lines), and true function (dotted line).} \label{fig:prod}}
\end{center}
\end{figure}

\begin{figure}[H]
\begin{center}
\includegraphics[width=6.5cm,height=6.0cm]{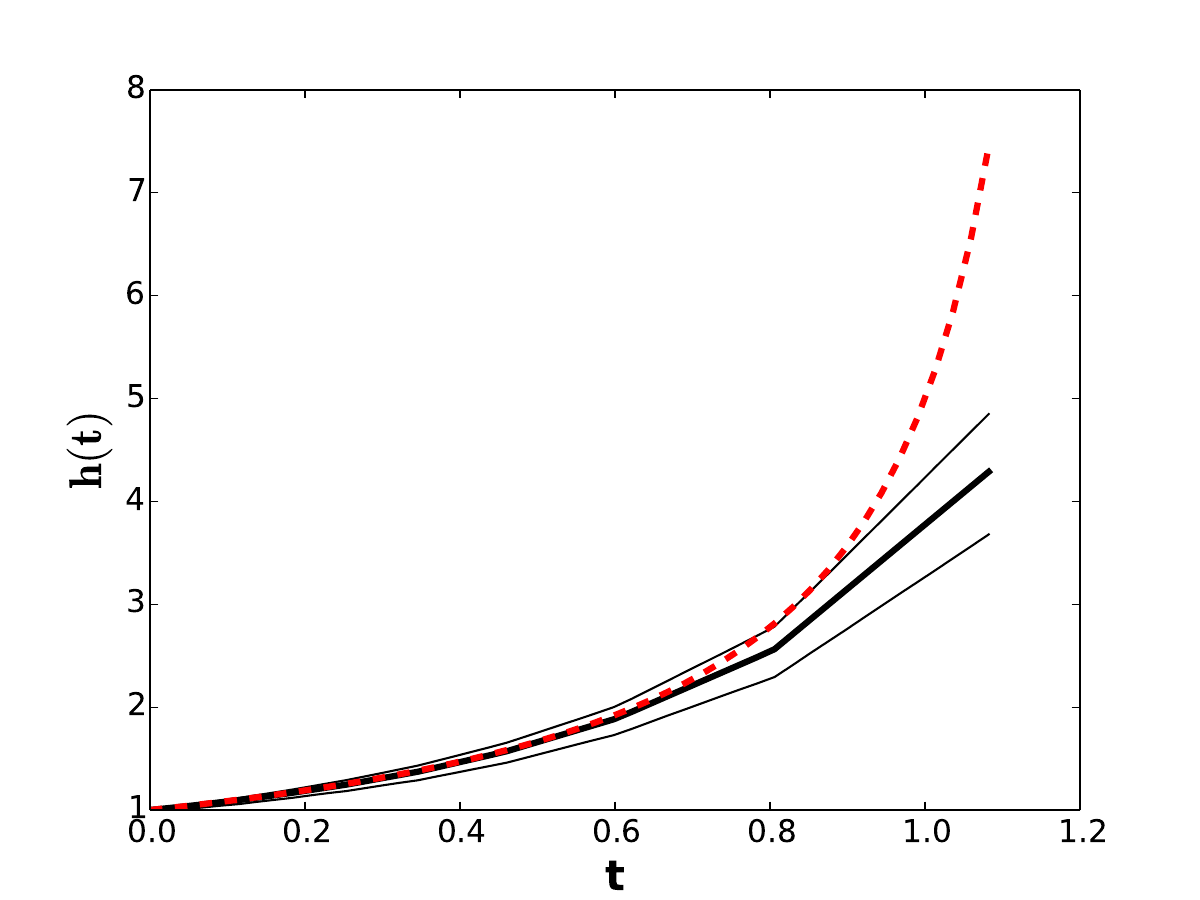}
\includegraphics[width=6.5cm,height=6.0cm]{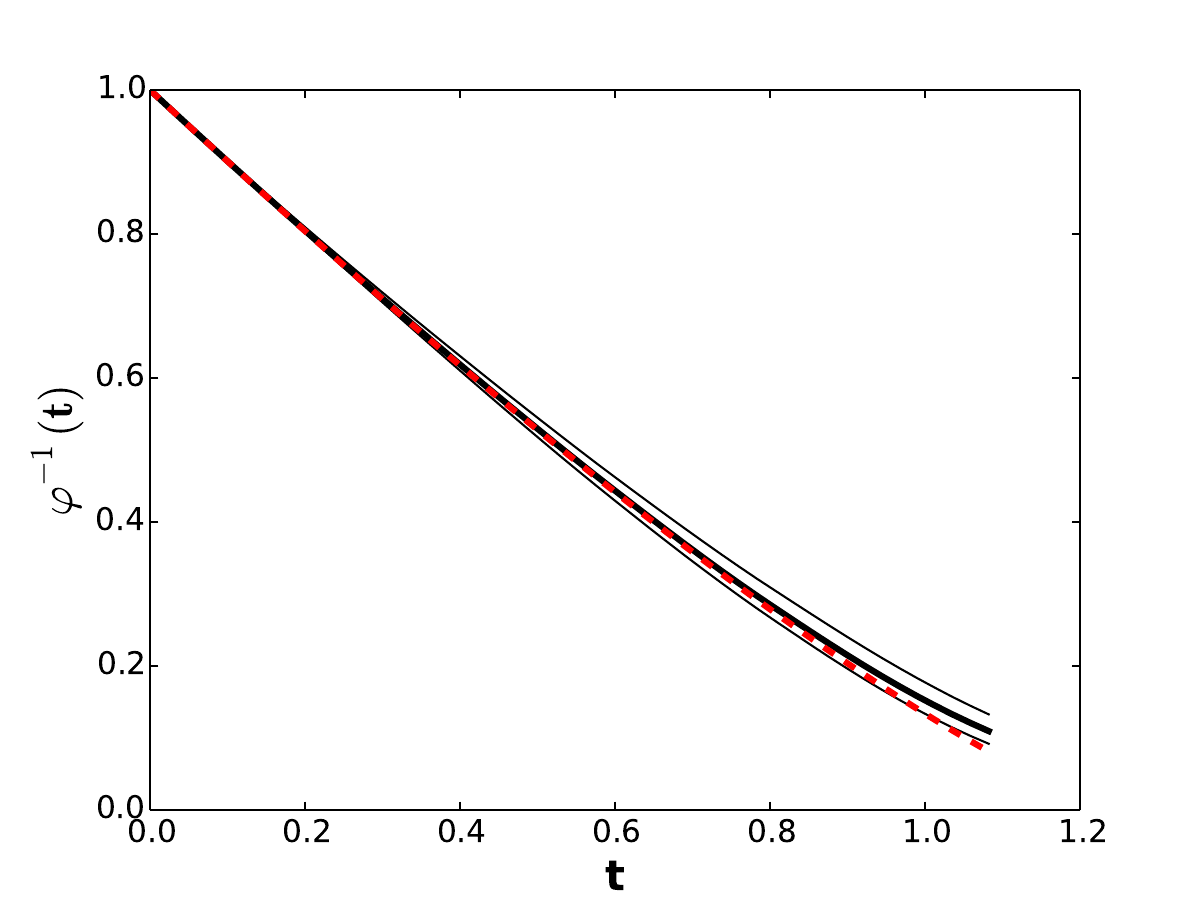}
\caption{\small{Posterior estimates of $h(t)$ and $\varphi^{-1}(t)$, obtained with a Log-$0.5$ partition of size $K=10$, for a simulated dataset of size $n=200$ from the Clayton copula with $\theta = -0.8$. Posterior mean (thick solid line), 95\% pointwise credible intervals (thin solid lines), and true function (dotted line).} \label{fig:clay_08}}
\end{center}
\end{figure}

\begin{figure}[H]
\begin{center}
\includegraphics[width=6.5cm,height=6.0cm]{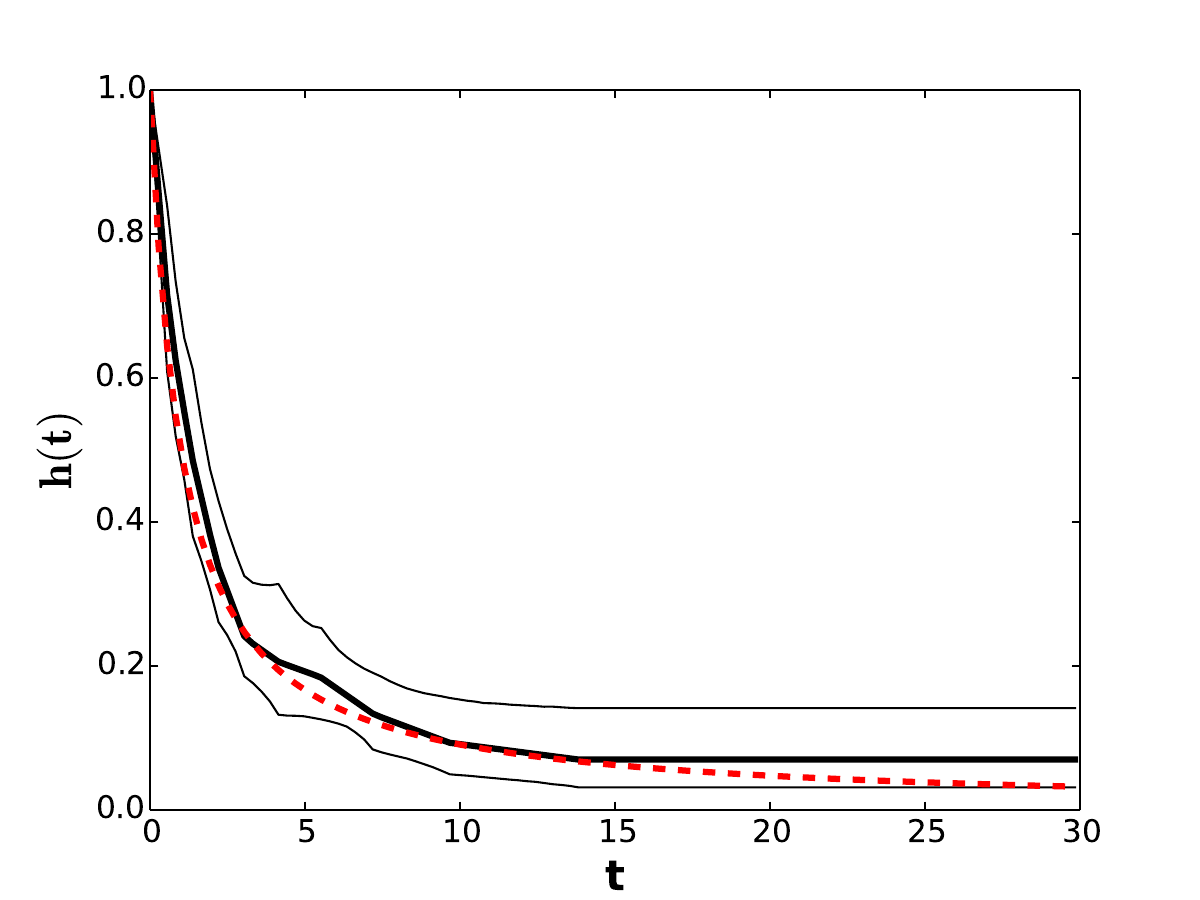}
\includegraphics[width=6.5cm,height=6.0cm]{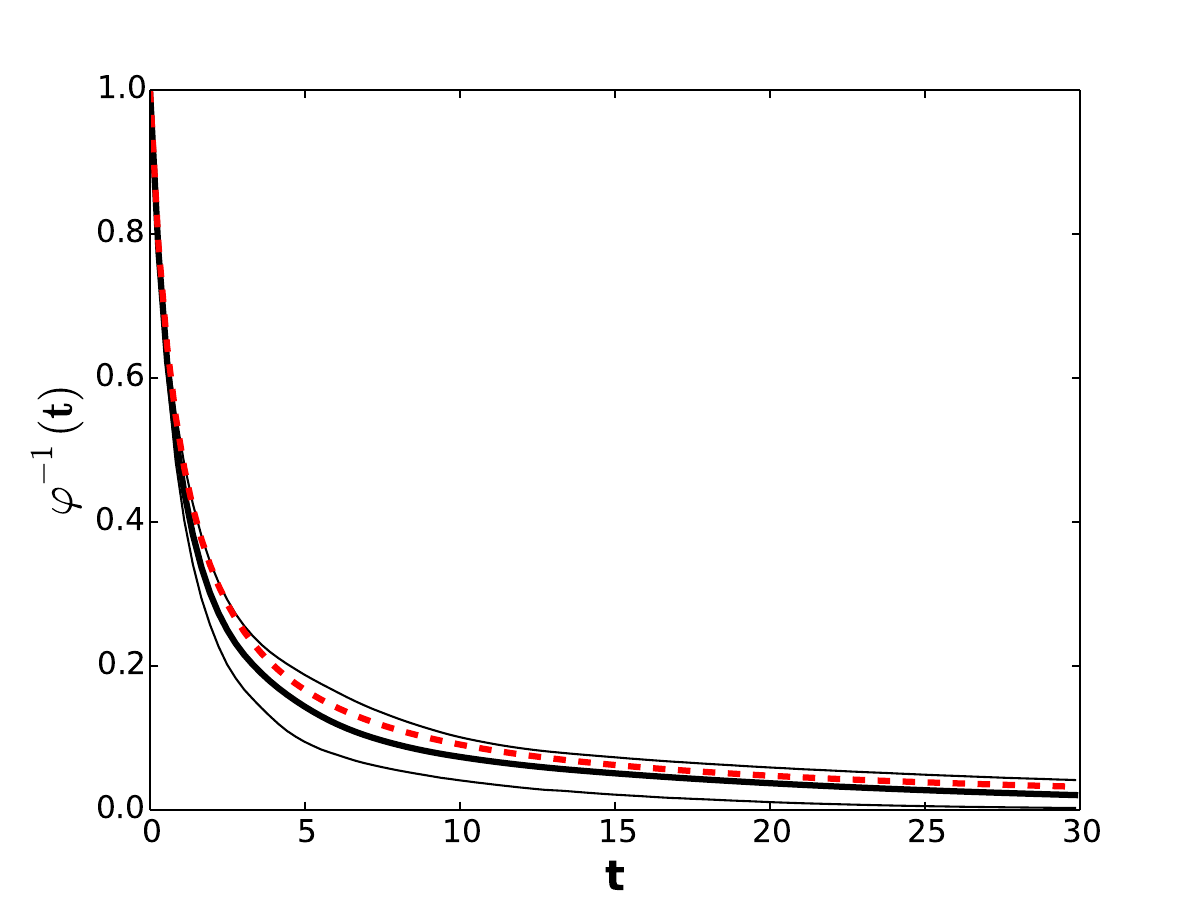}
\caption{\small{Posterior estimates of $h(t)$ and $\varphi^{-1}(t)$,  obtained with a Log-$6$ partition of size $K=10$, for a simulated dataset of size $n=200$ from the Clayton copula with $\theta = 1$. Posterior mean (thick solid line), 95\% pointwise credible intervals (thin solid lines), and true function (dotted line).} \label{fig:clay1}}
\end{center}
\end{figure}

\begin{figure}[H]
\begin{center}
\includegraphics[width=6.5cm,height=6.0cm]{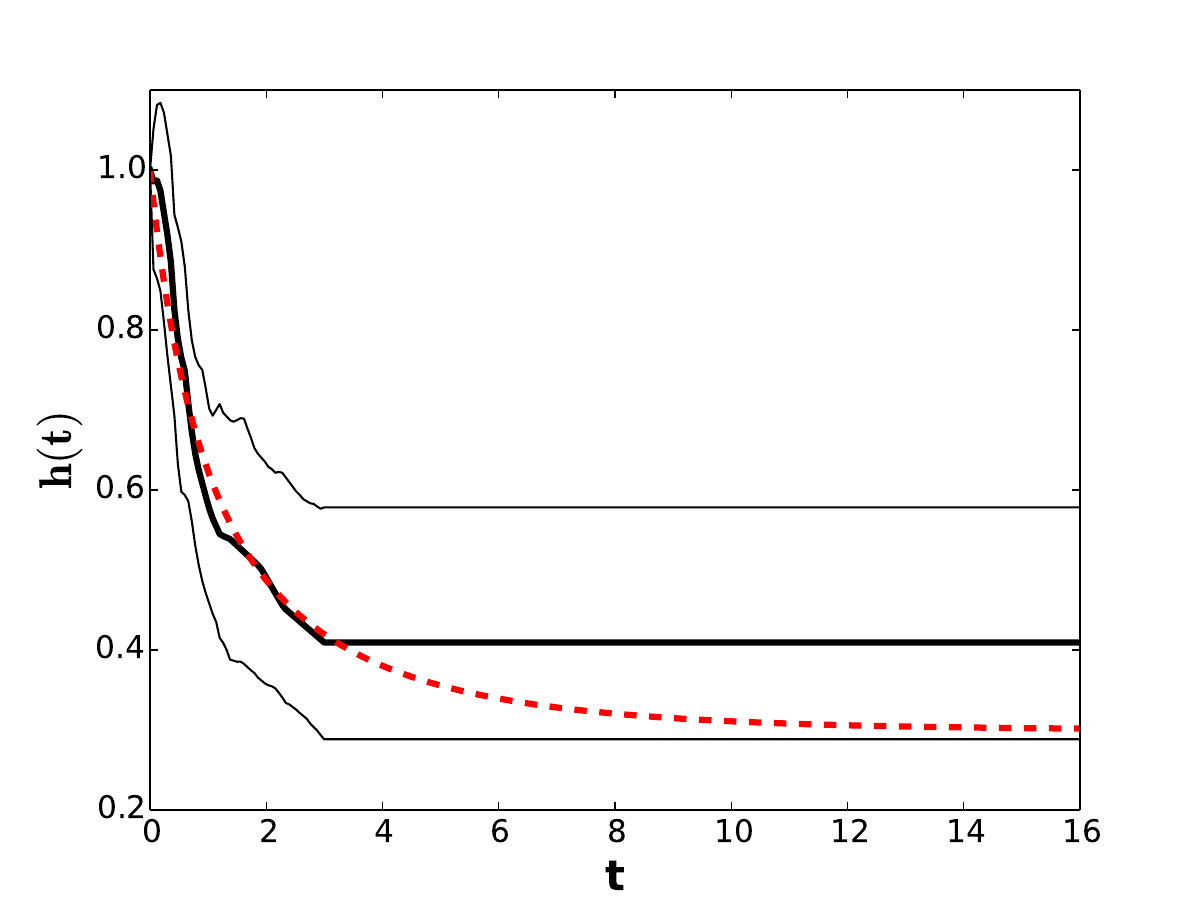}
\includegraphics[width=6.5cm,height=6.0cm]{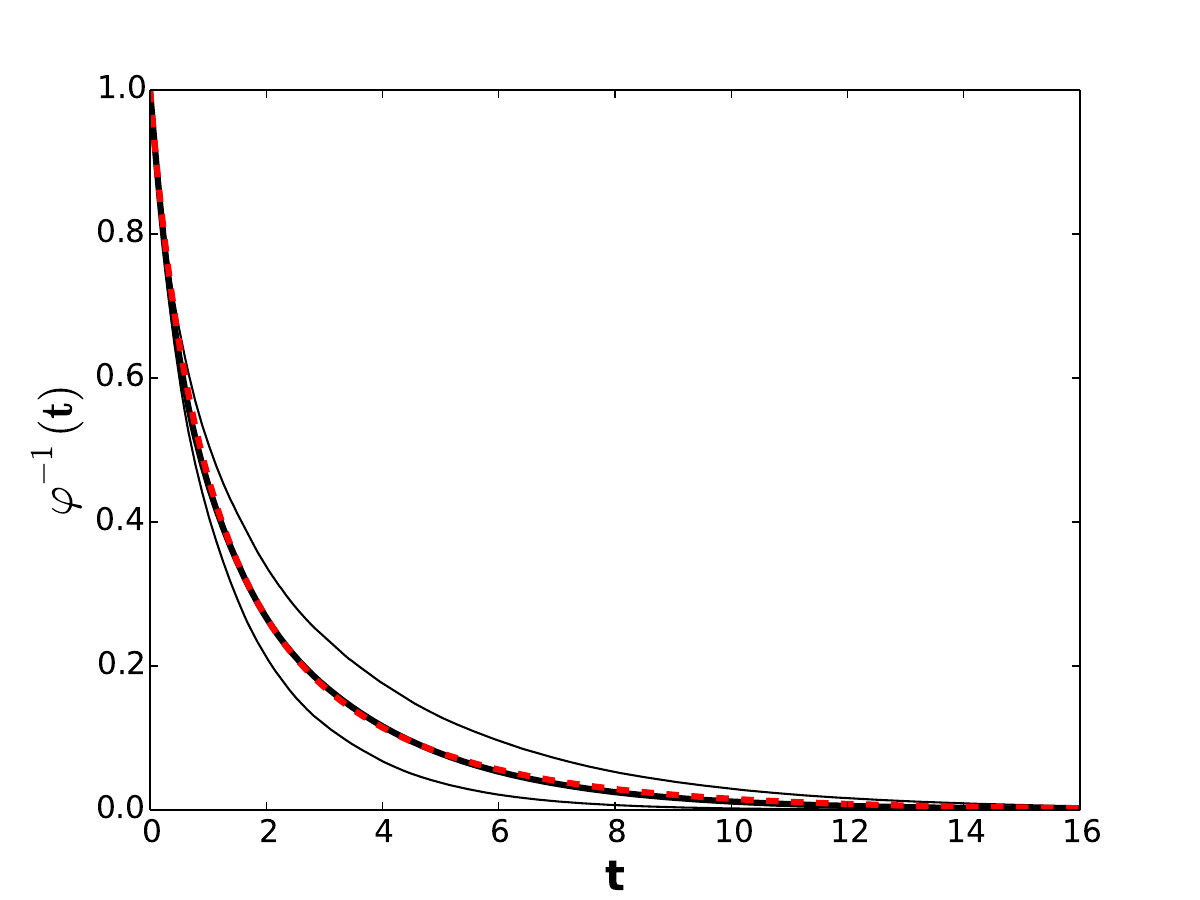}
\caption{\small{Posterior estimates of $h(t)$ and $\varphi^{-1}(t)$, obtained with a Log-$1$ partition of size $K=20$, for a simulated dataset of size $n=200$ from the AMH copula with $\theta = 0.7$. Posterior mean (thick solid line), 95\% pointwise credible intervals (thin solid lines), and true function (dotted line).} \label{fig:amh07}}
\end{center}
\end{figure}

\begin{figure}[H]
\begin{center}
\includegraphics[width=6.5cm,height=6.0cm]{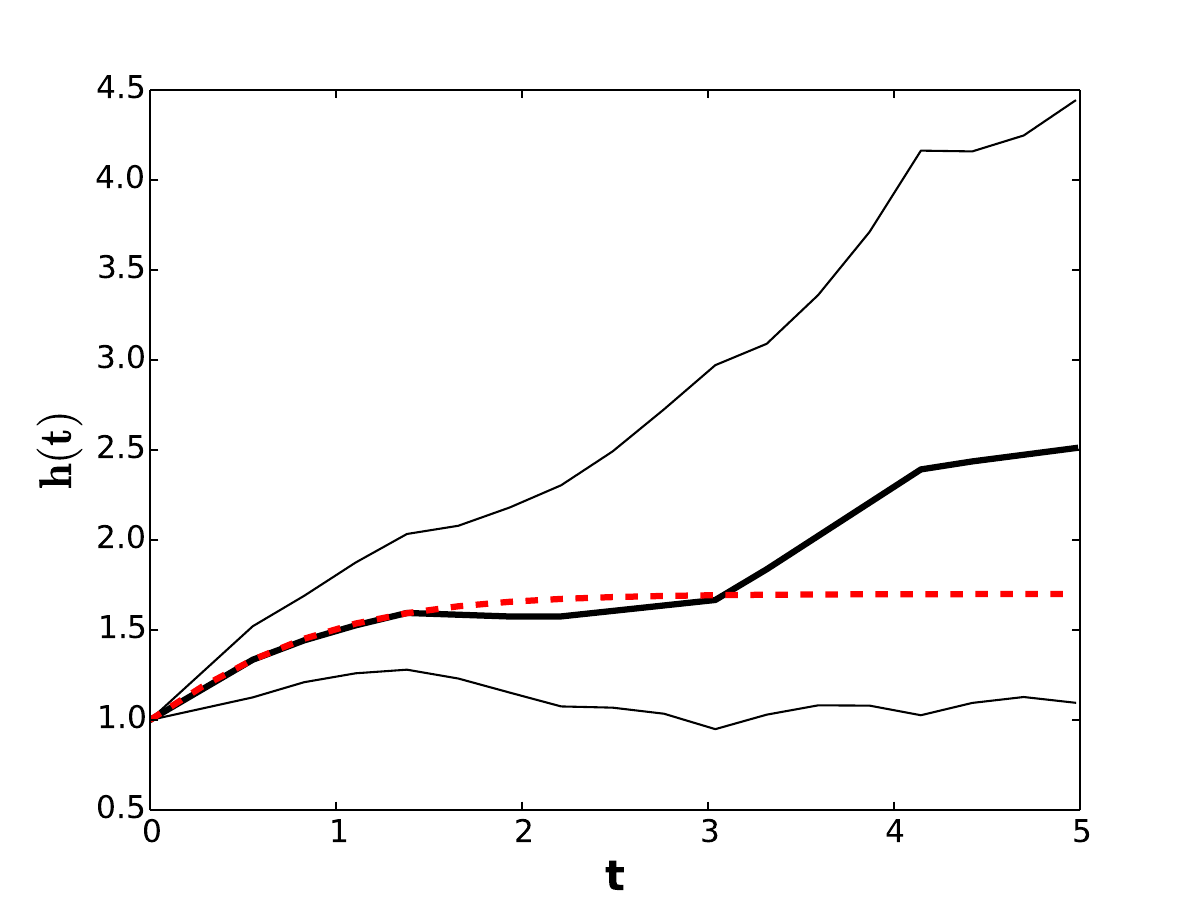}
\includegraphics[width=6.5cm,height=6.0cm]{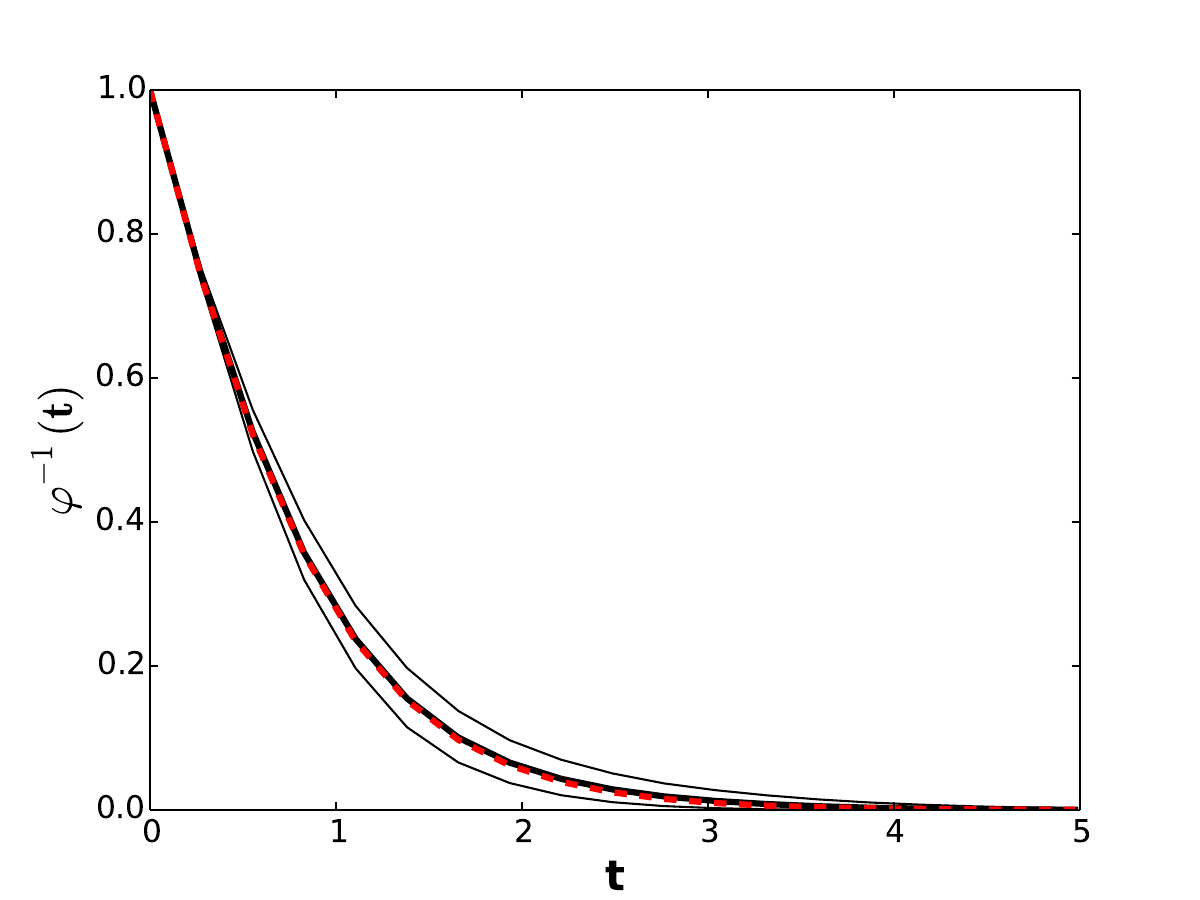}
\caption{\small{Posterior estimates of $h(t)$ and $\varphi^{-1}(t)$, obtained with a Log-$6$ partition of size $K=10$, for a simulated dataset of size $n=200$ from the AMH copula with $\theta = -0.7$. Posterior mean (thick solid line), 95\% pointwise credible intervals (thin solid lines), and true function (dotted line).} \label{fig:amh_07}}
\end{center}
\end{figure}

\begin{figure}[H]
\begin{center}
\includegraphics[width=6.5cm,height=6.0cm]{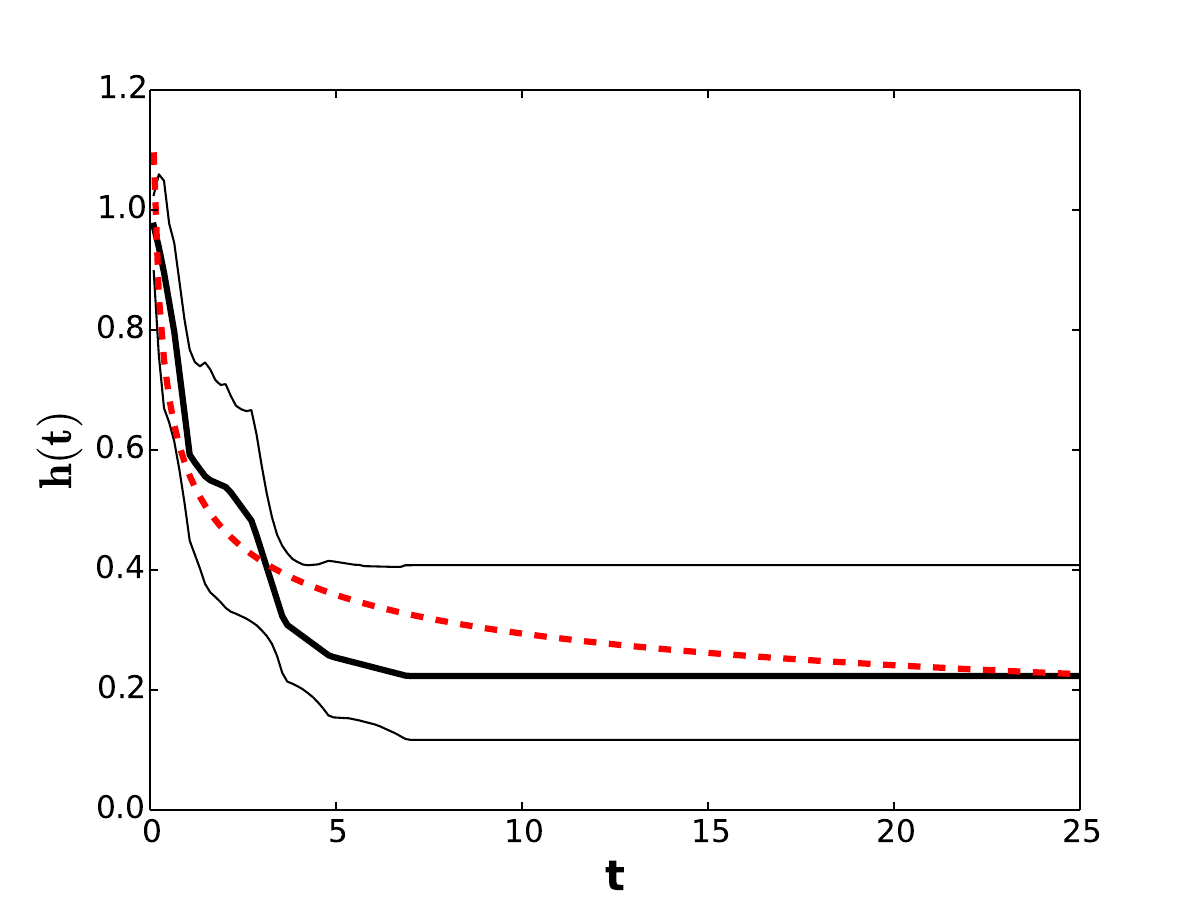}
\includegraphics[width=6.5cm,height=6.0cm]{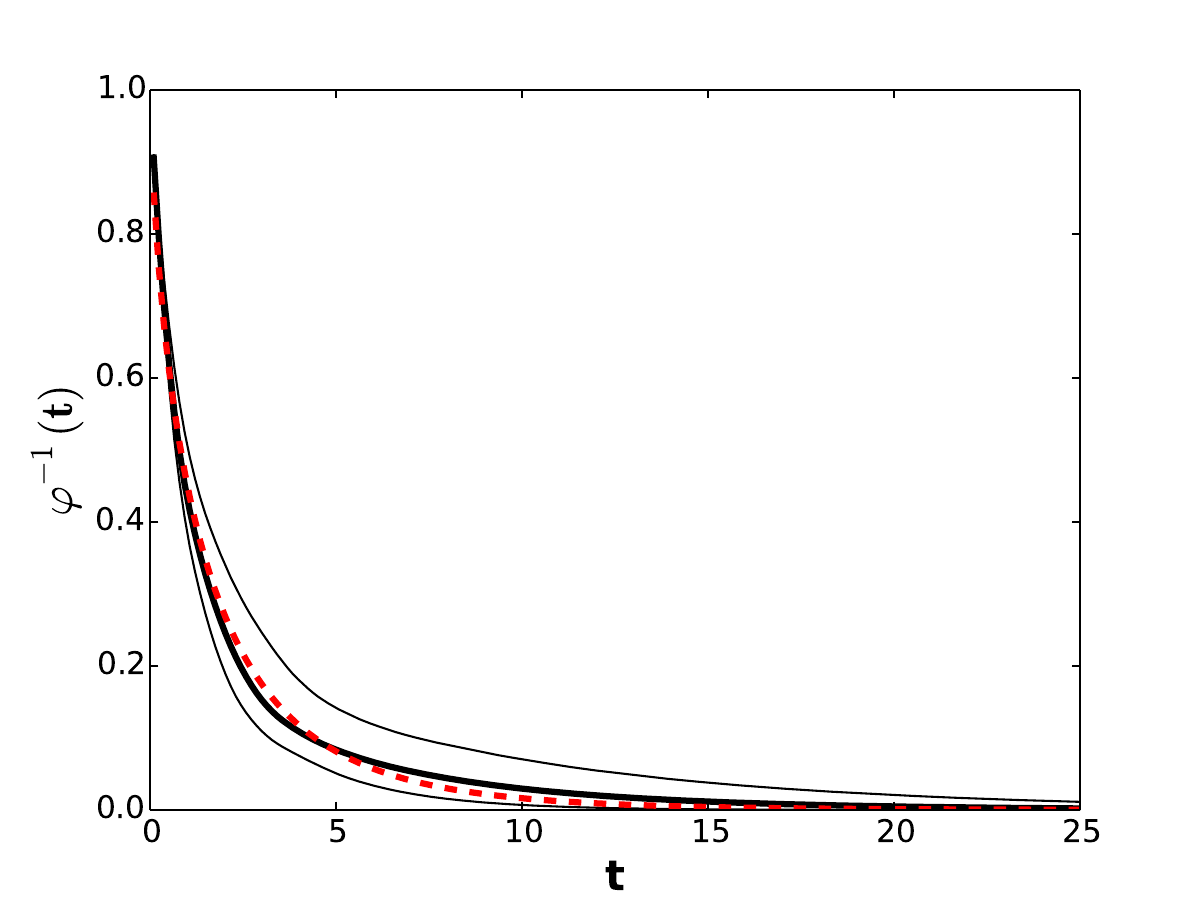}
\caption{\small{Posterior estimates of $h(t)$ and $\varphi^{-1}(t)$, obtained with a Log-$3$ partition of size $K=10$, for a simulated dataset of size $n=200$ from the Gumbel copula with $\theta = 1.4$. Posterior mean (thick solid line), 95\% pointwise credible intervals (thin solid lines), and true function (dotted line).} \label{fig:gumbel14}}
\end{center}
\end{figure}

\begin{figure}[H]
\begin{center}
\includegraphics[width=6.5cm,height=6.0cm]{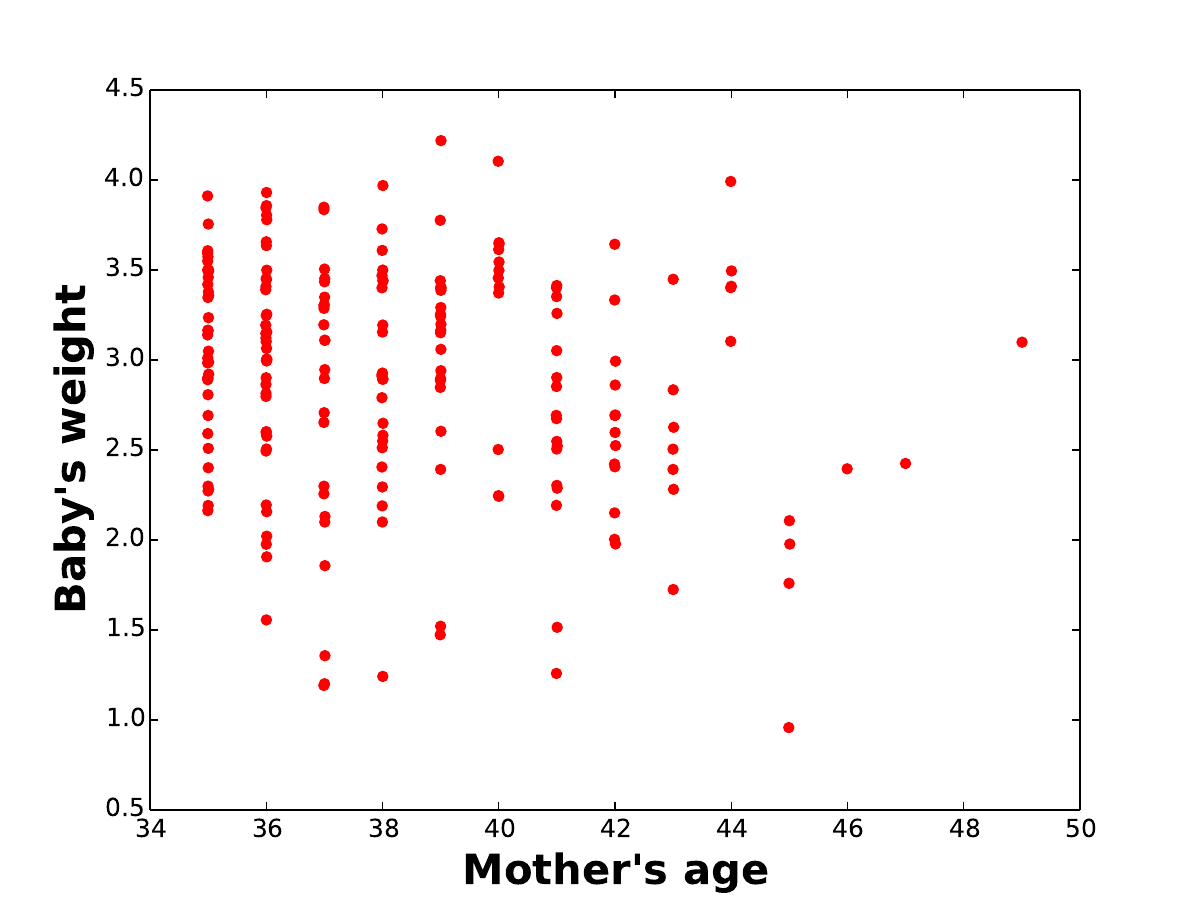}
\includegraphics[width=6.5cm,height=6.0cm]{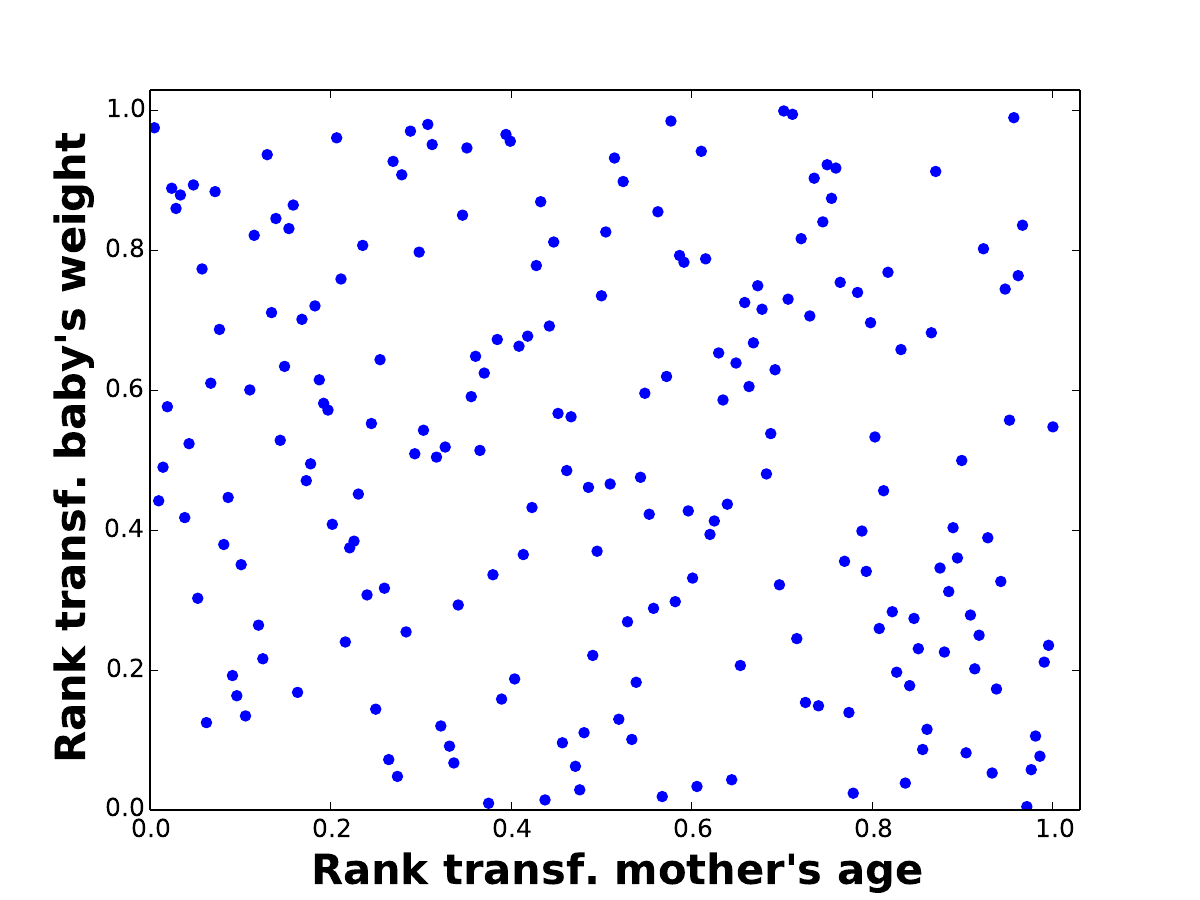}
\caption{\small{Scatter plots. Original data (left) and rank transformed data (right).} \label{fig:realdata}}
\end{center}
\end{figure}

\begin{figure}[H]
\begin{center}
\includegraphics[width=6.5cm,height=6.0cm]{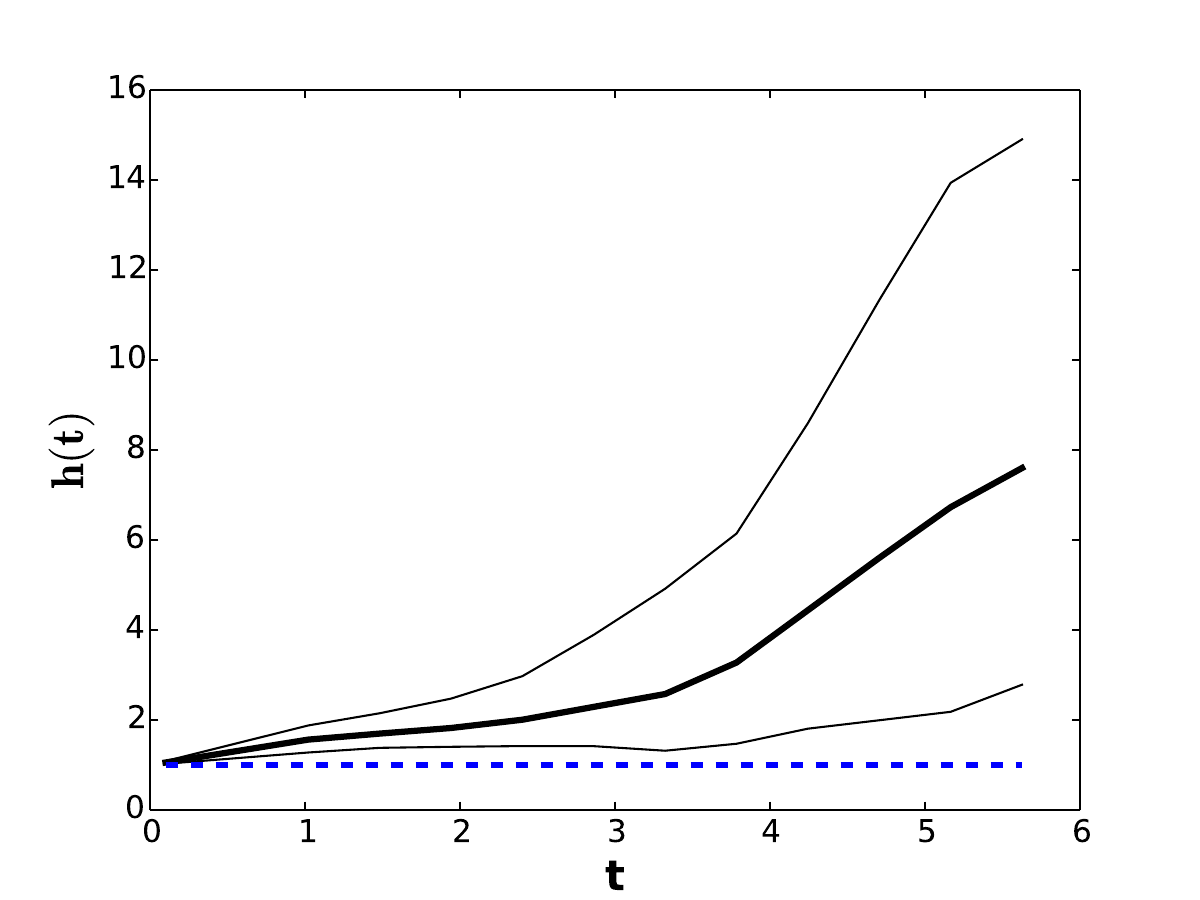}
\includegraphics[width=6.5cm,height=6.0cm]{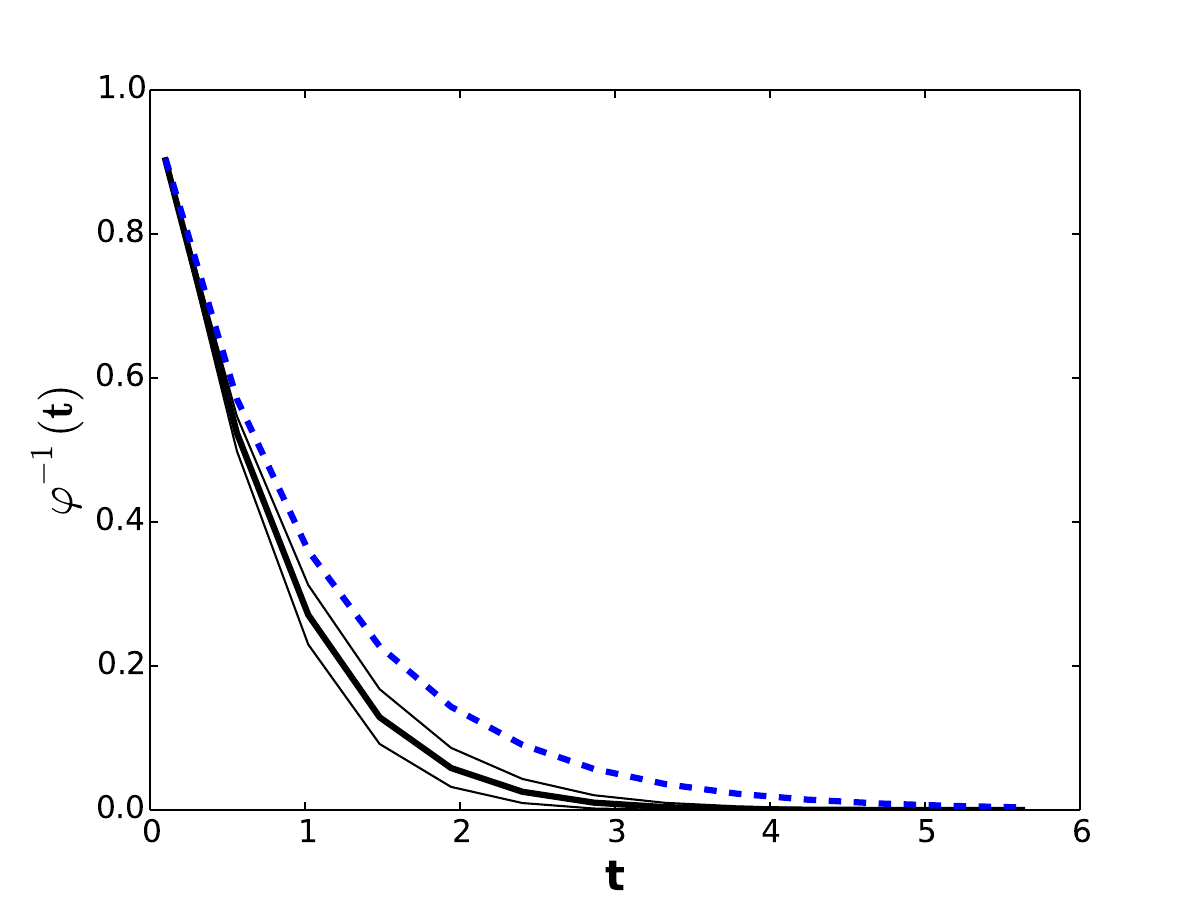}
\includegraphics[width=6.5cm,height=6.0cm]{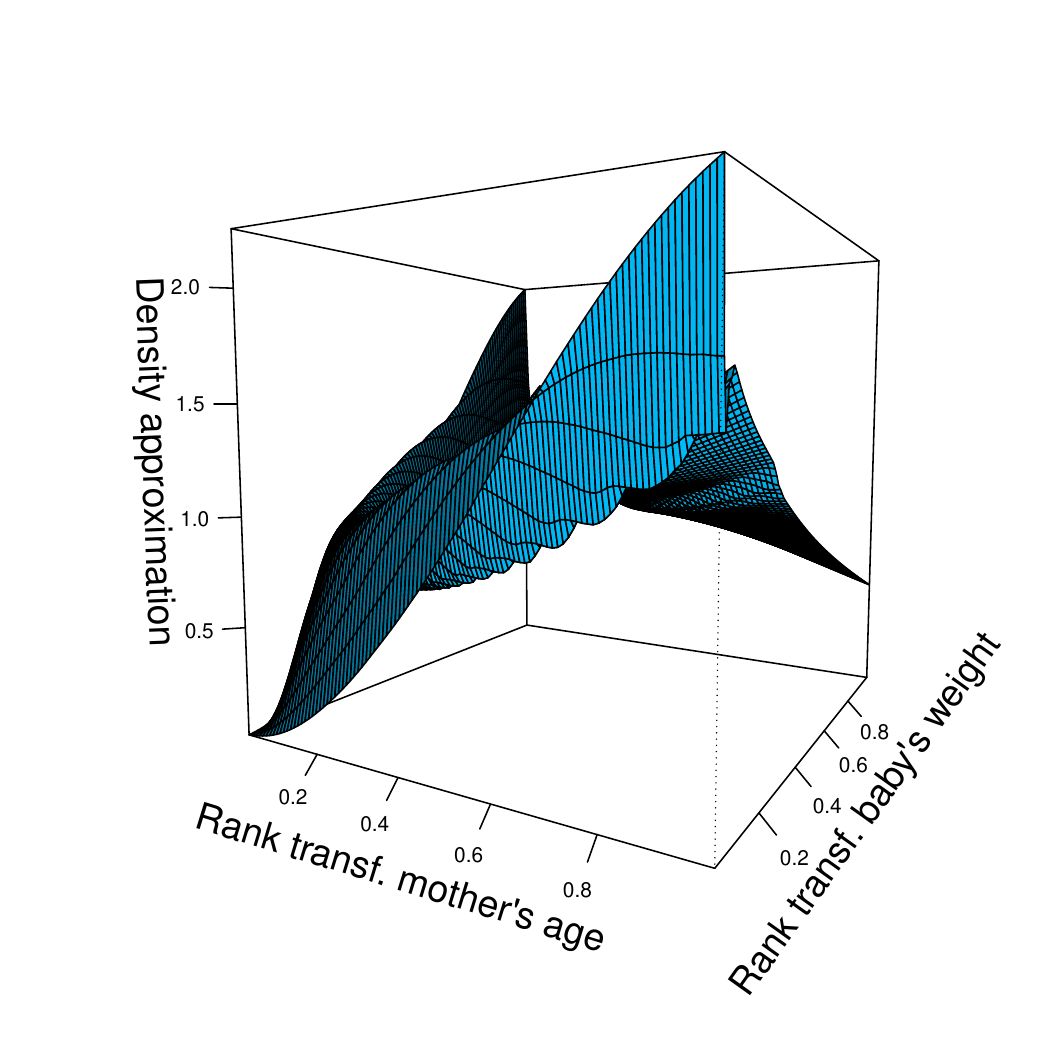}
\includegraphics[width=6.5cm,height=6.0cm]{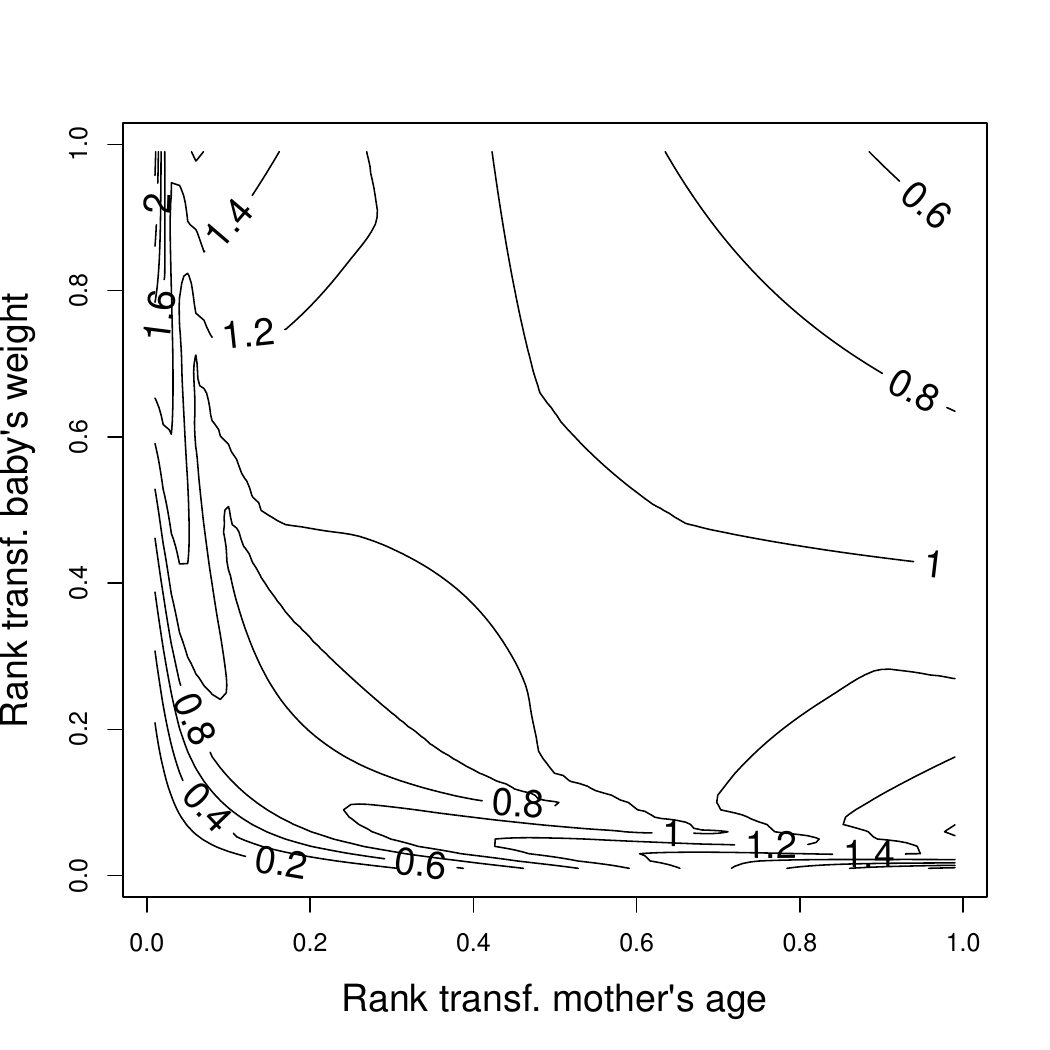}
\caption{\small{Posterior estimates obtained with a Log-$10$ partition of size $K=10$ for the real dataset. Top row: $h(t)$ and $\varphi^{-1}(t)$,
posterior mean (thick solid line) and 95\% pointwise credible intervals (thin solid lines). Corresponding functions from the product copula (dotted lines) are included for visual comparison. Bottom row: joint density and contour plot estimates.} \label{fig:perealdata}}
\end{center}
\end{figure}

\begin{table}[H]
\caption{Summary of some parametric Archimedean copulas parametrised such that $h(0)=1$, for the first three copulas, and $h(\epsilon)=1$, for the Gumbel copula.}
\label{tab:parcopulas}
\begin{center}
\begin{tabular}{cccccc}
\hline \hline
Copula & $\Theta$ & $\varphi(t)$ & $\varphi^{-1}(t)$ & Strict? & h(t)  \\
\hline 
Product & -  & $ - \log(t)$ & $e^{-t}$ & Yes & $1$  \\ 
Clayton & $[-1,\infty)$  & $\frac{1}{\theta}(t^{- \theta}-1)$ & $ (1+ \theta t)^{-1/\theta}$ & If $\theta \geq 0$ & $\frac{1}{1 + \theta t}$   \\ 
AMH & $[-1,1)$  & $\frac{1}{1-\theta}\log \left( \frac{1 - \theta + \theta t}{t} \right)$ & $ \frac{1- \theta}{e^{(1-\theta)t} - \theta}$ & Yes & $\frac{(1-\theta)e^{(1-\theta)t}}{e^{(1-\theta)t} - \theta}$ \\ 
Gumbel & $[1,\infty)$  & $ \epsilon\left(\frac{-\log(t)}{\epsilon\theta}\right)^{\theta}$ & $\exp\left\{-\epsilon\theta\left(\frac{t}{\epsilon}\right)^{1/\theta}\right\}$ & Yes & $\left(\frac{\epsilon}{t}\right)^{1- 1/\theta}$  \\ \hline \hline
\end{tabular}
\end{center}
\end{table}

\begin{table}[H]
\caption{GOF measures, obtained with different partition definitions, for a simulated dataset of size $n=200$ from the product copula.}
\label{tab:prod}
\normalsize{
\sisetup{round-mode=places,round-precision=3}
\begin{center}
\begin{tabular}{cccccccc}\hline \hline
Part.Type & $K$ & $Q_{\kappa_\tau}^{(0.025)}$ & $\hat{\kappa}_\tau$ & $Q_{\kappa_\tau}^{(0.975)}$ & $\kappa_\tau$ &  Sup.Norm & LPML \\
\hline 
\begin{filecontents*}{ResumenProducto.csv}
Type,Size,25Tau,Mtau,975Tau,RealTau,Norm,LPML
Log-1,10,-0.1022689136,-0.0290909778,0.0558081751,0,0.0447948335,-2.0706900864
Log-2,10,-0.1215696759,-0.0536237027,0.0127699259,0,0.0610216867,-2.2100032736
Log-3,10,-0.1572180303,-0.0716449801,0.0068993167,0,0.0612988493,-4.4612796536
Log-1,20,-0.0827512602,-0.0045139048,0.0658924723,0,0.0361858911,-4.8954478076
Log-2,20,-0.1377246658,-0.0640720821,0.0021076309,0,0.0647699288,-2.7440451131
Log-3,20,-0.1409709989,-0.0731029628,-0.0065641467,0,0.0635527293,-3.0507455458
\end{filecontents*}
\csvreader[late after line=\\]{ResumenProducto.csv}{Type=\Type, Size=\Size, 25Tau=\25Tau, Mtau=\Mtau, 975Tau=\975Tau, RealTau=\RealTau, Norm=\Norm, LPML=\LPML}%
{\Type & \Size & \num{\25Tau} & \num{\Mtau} & \num{\975Tau} & \num{\RealTau}  & \num{\Norm} & \num{\LPML}}%
\hline \hline
\end{tabular}
\end{center}
}\end{table}

\begin{table}[H]
\caption{GOF measures, obtained with different partition definitions, for a simulated dataset of size $n=200$ from the Clayton copula with $\theta = -0.8$.}
\label{tab:clay_08}
\normalsize{
\sisetup{round-mode=places,round-precision=3}
\begin{center}
\begin{tabular}{cccccccc}\hline \hline
Part.Type & $K$ & $Q_{\kappa_\tau}^{(0.025)}$ & $\hat{\kappa}_\tau$ & $Q_{\kappa_\tau}^{(0.975)}$ & $\kappa_\tau$ &  Sup.Norm & LPML \\
\hline 
\begin{filecontents*}{ResumenClayton-08.csv}
Type,Size,25Tau,Mtau,975Tau,RealTau,Norm,LPML
Log-0.3,10,-0.261791067,-0.2069531909,-0.1419261747,-0.6666,0.1316468208,29.1217298406
Log-0.5,10,-0.5099783782,-0.4776352025,-0.4358174702,-0.6666,0.0275092758,112.249382355
Log-0.9,10,-0.4868390533,-0.4514681821,-0.4086189817,-0.6666,0.1191123514,136.807266111
Log-0.3,20,-0.359316377,-0.3045391407,-0.2464683289,-0.6666,0.1039329764,49.2447157278
Log-0.5,20,-0.5162995671,-0.4763164018,-0.4290749724,-0.6666,0.1403850435,141.270798127
Log-0.9,20,-0.5074682822,-0.4608106753,-0.4230185926,-0.6666,0.1221225066,152.163346021
\end{filecontents*}
\csvreader[late after line=\\]{ResumenClayton-08.csv}{Type=\Type, Size=\Size, 25Tau=\25Tau, Mtau=\Mtau, 975Tau=\975Tau, RealTau=\RealTau, Norm=\Norm, LPML=\LPML}%
{\Type & \Size & \num{\25Tau} & \num{\Mtau} & \num{\975Tau} & \num{\RealTau}  & \num{\Norm} & \num{\LPML}}%
\hline \hline
\end{tabular}
\end{center}
}\end{table}

\begin{table}[H]
\caption{GOF measures, obtained with different partition definitions, for a simulated dataset of size $n=200$ from the Clayton copula with $\theta = 1$.}
\label{tab:clay1}
\normalsize{
\sisetup{round-mode=places,round-precision=3}
\begin{center}
\begin{tabular}{cccccccc}\hline\hline
Part.Type & $K$ & $Q_{\kappa_\tau}^{(0.025)}$ & $\hat{\kappa}_\tau$ & $Q_{\kappa_\tau}^{(0.975)}$ & $\kappa_\tau$ &  Sup.Norm & LPML \\
\hline 
\begin{filecontents*}{ResumenClayton1.csv}
Type,Size,25Tau,Mtau,975Tau,RealTau,Norm,LPML
Log-4,10,0.1976638923,0.2647004636,0.3155398189,0.3333,0.0650162898,32.9825654685
Log-6,10,0.221566028,0.2930516663,0.3503266962,0.3333,0.0451845059,34.74447765
Log-8,10,0.1742394885,0.2396622378,0.295213598,0.3333,0.0821039221,31.5844350832
Log-4,20,0.0845989962,0.2469494513,0.3231918405,0.3333,0.0923077988,27.5944369653
Log-6,20,0.043220304,0.1414299476,0.2153426196,0.3333,0.1441301361,15.3394323004
Log-8,20,0.0147521423,0.0778524925,0.1392699326,0.3333,0.1670343706,5.3258484167
\end{filecontents*}
\csvreader[late after line=\\]{ResumenClayton1.csv}{Type=\Type, Size=\Size, 25Tau=\25Tau, Mtau=\Mtau, 975Tau=\975Tau, RealTau=\RealTau, Norm=\Norm, LPML=\LPML}%
{\Type & \Size & \num{\25Tau} & \num{\Mtau} & \num{\975Tau} & \num{\RealTau}  & \num{\Norm} & \num{\LPML}}%
\hline \hline
\end{tabular}
\end{center}
}\end{table}

\begin{table}[H]
\caption{GOF measures, obtained with different partition definitions, for a simulated dataset of size $n=200$ from the AMH copula with $\theta = -0.7$.}
\label{tab:amh_07}
\normalsize{
\sisetup{round-mode=places,round-precision=3}
\begin{center}
\begin{tabular}{cccccccc}\hline\hline
Part.Type & $K$ & $Q_{\kappa_\tau}^{(0.025)}$ & $\hat{\kappa}_\tau$ & $Q_{\kappa_\tau}^{(0.975)}$ & $\kappa_\tau$ &  Sup.Norm & LPML \\
\hline 
\begin{filecontents*}{ResumenAMH-07.csv}
Type,Size,25Tau,Mtau,975Tau,RealTau,Norm,LPML
Log-4,10,-0.1965065962,-0.1334339658,-0.0495029429,-0.134,0.0220484285,1.7332078811
Log-6,10,-0.2001244419,-0.1320571607,-0.0697999333,-0.134,0.0214542285,2.8685105953
Log-8,10,-0.204640631,-0.1470139845,-0.0790921467,-0.134,0.021874576,2.6949373458
Log-4,20,-0.2081902112,-0.139224386,-0.0674782988,-0.134,0.021998123,-0.8648945592
Log-6,20,-0.204908849,-0.1376268058,-0.0672020158,-0.134,0.0205609853,0.7793308988
Log-8,20,-0.2065562799,-0.1369274676,-0.0738343755,-0.134,0.0192696907,1.8278372644
\end{filecontents*}
\csvreader[late after line=\\]{ResumenAMH-07.csv}{Type=\Type, Size=\Size, 25Tau=\25Tau, Mtau=\Mtau, 975Tau=\975Tau, RealTau=\RealTau, Norm=\Norm, LPML=\LPML}%
{\Type & \Size & \num{\25Tau} & \num{\Mtau} & \num{\975Tau} & \num{\RealTau}  & \num{\Norm} & \num{\LPML}}%
\hline \hline
\end{tabular}
\end{center}
}\end{table}

\begin{table}[H]
\caption{GOF measures, obtained with different partition definitions, for a simulated dataset of size $n=200$ from the AMH copula with $\theta = 0.7$.}
\label{tab:amh07}
\normalsize{
\sisetup{round-mode=places,round-precision=3}
\begin{center}
\begin{tabular}{cccccccc}\hline\hline
Part.Type & $K$ & $Q_{\kappa_\tau}^{(0.025)}$ & $\hat{\kappa}_\tau$ & $Q_{\kappa_\tau}^{(0.975)}$ & $\kappa_\tau$ &  Sup.Norm & LPML \\
\hline 
\begin{filecontents*}{ResumenAMH07.csv}
Type,Size,25Tau,Mtau,975Tau,RealTau,Norm,LPML
Log-1,10,0.1265490444,0.1933537176,0.2489927374,0.1954,0.0300776039,6.7206238893
Log-3,10,0.0466524043,0.1225383542,0.2047499672,0.1954,0.0684807927,4.2207621984
Log-7,10,0.0062256323,0.0880905959,0.169009927,0.1954,0.0707821802,0.8470353249
Log-1,20,0.1389717246,0.2034274851,0.2636791949,0.1954,0.0331362076,4.5240801179
Log-3,20,-0.0071099655,0.0877189533,0.1878928411,0.1954,0.0939002118,-2.0828832998
Log-7,20,-0.0365527983,0.0182710248,0.0850276439,0.1954,0.130235056,-5.7152642273
\end{filecontents*}
\csvreader[late after line=\\]{ResumenAMH07.csv}{Type=\Type, Size=\Size, 25Tau=\25Tau, Mtau=\Mtau, 975Tau=\975Tau, RealTau=\RealTau, Norm=\Norm, LPML=\LPML}%
{\Type & \Size & \num{\25Tau} & \num{\Mtau} & \num{\975Tau} & \num{\RealTau}  & \num{\Norm} & \num{\LPML}}%
\hline \hline
\end{tabular}
\end{center}
}\end{table}

\begin{table}[H]
\caption{GOF measures, obtained with different partition definitions, for a simulated dataset of size $n=200$ from the Gumbel copula with $\theta = 1.4$.}
\label{tab:gumbel14}
\normalsize{
\sisetup{round-mode=places,round-precision=3}
\begin{center}
\begin{tabular}{cccccccc}\hline\hline
Part.Type & $K$ & $Q_{\kappa_\tau}^{(0.025)}$ & $\hat{\kappa}_\tau$ & $Q_{\kappa_\tau}^{(0.975)}$ & $\kappa_\tau$ &  Sup.Norm & LPML \\
\hline 
\begin{filecontents*}{ResumenGumbel1_4.csv}
Type,Size,25Tau,Mtau,975Tau,RealTau,Norm,LPML
Log-1,10,0.1648007448,0.2386401007,0.2944043051,0.2857142857,0.0827370102,14.0248793493
Log-3,10,0.1291889796,0.197415633,0.2903136689,0.2857142857,0.0553034148,18.013033536
Log-7,10,0.0960290753,0.1345131414,0.1530697038,0.2857142857,0.0635395165,21.3832211273
Log-1,20,0.0729263192,0.1604539548,0.221657027,0.2857142857,0.0636531342,11.9907707317
Log-3,20,0.0338275358,0.1053023242,0.2490627644,0.2857142857,0.0910729742,16.9081059457
Log-7,20,-0.0093929597,0.0626107981,0.1311486182,0.2857142857,0.1121652419,11.1340790493
\end{filecontents*}
\csvreader[late after line=\\]{ResumenGumbel1_4.csv}{Type=\Type, Size=\Size, 25Tau=\25Tau, Mtau=\Mtau, 975Tau=\975Tau, RealTau=\RealTau, Norm=\Norm, LPML=\LPML}%
{\Type & \Size & \num{\25Tau} & \num{\Mtau} & \num{\975Tau} & \num{\RealTau}  & \num{\Norm} & \num{\LPML}}%
\hline \hline
\end{tabular}
\end{center}
}\end{table}

\begin{table}[H]
\caption{GOF measures, obtained with different partition definitions, for  the real data.}
\label{tab:realdata}
\normalsize{
\sisetup{round-mode=places,round-precision=3}
\begin{center}
\begin{tabular}{ccccccc}\hline\hline
Part.Type & $K$ & $Q_{\kappa_\tau}^{(0.025)}$ & $\hat{\kappa}_\tau$ & $Q_{\kappa_\tau}^{(0.975)}$ & Sample $\tilde{\kappa}_\tau$ &  LPML \\
\hline
\begin{filecontents*}{ResumenDatosb.csv}
Type,Size,25Tau,Mtau,975Tau,RealTau,Norm,LPML
Log-6,10,-0.2119679298,-0.152788183,-0.0915523204,-0.1162,NA,2.9667539189
Log-8,10,-0.2195007995,-0.1531655477,-0.0867252703,-0.1162,NA,2.2936525781
Log-10,10,-0.2076680769,-0.14958862,-0.0941137668,-0.1162,NA,3.7016248425
Log-6,20,-0.2117377276,-0.1527517752,-0.0907067361,-0.1162,NA,0.4614225103
Log-8,20,-0.2173709305,-0.1500547571,-0.0973730775,-0.1162,NA,1.3033744634
Log-10,20,-0.2310610418,-0.1580465112,-0.0959541568,-0.1162,NA,2.0458296122
\end{filecontents*}
\csvreader[late after line=\\]{ResumenDatosb.csv}{Type=\Type, Size=\Size, 25Tau=\25Tau, Mtau=\Mtau, 975Tau=\975Tau, RealTau=\RealTau, Norm=\Norm, LPML=\LPML}%
{\Type & \Size & \num{\25Tau} & \num{\Mtau} & \num{\975Tau} & \num{\RealTau} & \num{\LPML}}%
\hline\hline
\end{tabular}
\end{center}
}\end{table}

\begin{table}[H]
\caption{Independence test for all simulated and real datasets. semiparametric Archimedean copula test (SPAC), contingency table test (CT), mixture model test (MM), and empirical copula test (EC). For the Bayesian tests we report $\P(H_1\mid\data)$ and for the frequentist test we report p-values.}
\label{tab:indep}
\normalsize{
\begin{center}
\begin{tabular}{cr|ccc|c}\hline\hline
 & & \multicolumn{3}{c|}{Bayesian Tests} & Freq.Test \\
Dataset & $\theta$ & SPAC & CT & MM & EC \\
\hline
Product & -- & 0.19 & 0.04 & 0.12 & 0.15 \\
Clayton & $-0.8$ & 1.00 & 1.00 & 0.99 & 0.00 \\
Clayton & $-0.4$ & 0.96 & 0.65 & 0.83 & 0.00 \\
Clayton & $0.6$ & 0.96 & 0.43 & 0.79 & 0.01 \\
Clayton & $1.0$ & 1.00 & 0.87 & 0.75 & 0.00 \\
AMH & $-0.7$ & 0.63 & 0.11 & 0.21 & 0.31 \\
AMH & $-0.3$ & 0.23 & 0.05 &  0.08 & 0.79 \\
AMH & $0.3$ & 0.22 & 0.06 & 0.09 & 0.81 \\
AMH & $0.7$ & 0.86 & 0.38 & 0.48 & 0.33 \\
Gumbel & $1.4$ & 1.00 & 0.50 & 0.75 & 0.00 \\
Gumbel &$2.0$ & 1.00 & 0.97 & 0.93 & 0.00 \\
Real data & -- & 0.47 & 0.07 & 0.40 & 0.35 \\
\hline\hline
\end{tabular}
\end{center}
}\end{table}

\end{document}